\documentclass[bj,noshowframe,numbers]{imsart}

\usepackage{etoolbox}
\newcommand{\CCov}{\mathcal{C}}

\newcommand{\bbR}{\mathbb{R}}

\newcommand{\sN}{\mathcal{N}}
\newcommand{\sM}{\mathcal{M}}

\newcommand{\sT}{\mathcal{T}}

\newcommand{\sU}{\mathcal{U}}

\newcommand{\vb}{\mathbf{b}}

\newcommand{\vu}{\mathbf{u}}

\ifdef{\vv}{
\renewcommand{\vv}{\mathbf{v}}
}{
\newcommand{\vv}{\mathbf{v}}
}

\newcommand{\vx}{\mathbf{x}}
\newcommand{\vy}{\mathbf{y}}

\newcommand{\vzero}{\mathbf{0}}

\newcommand{\vbeta}{\boldsymbol{\beta}}
\newcommand{\mA}{\mathbf{A}}
\newcommand{\mB}{\mathbf{B}}
\newcommand{\mC}{\mathbf{C}}
\newcommand{\mD}{\mathbf{D}}

\newcommand{\mF}{\mathbf{F}}
\newcommand{\mG}{\mathbf{G}}
\newcommand{\mH}{\mathbf{H}}
\newcommand{\mI}{\mathbf{I}}

\newcommand{\mK}{\mathbf{K}}

\newcommand{\mN}{\mathbf{N}}

\newcommand{\mQ}{\mathbf{Q}}
\newcommand{\mR}{\mathbf{R}}
\newcommand{\mS}{\mathbf{S}}

\newcommand{\mU}{\mathbf{U}}

\newcommand{\mW}{\mathbf{W}}
\newcommand{\mX}{\mathbf{X}}
\newcommand{\mY}{\mathbf{Y}}
\newcommand{\mZ}{\mathbf{Z}}
\newcommand{\mGamma}{\boldsymbol{\Gamma}}

\newcommand{\mDelta}{\boldsymbol{\Delta}}

\newcommand{\bmD}{\bar{\mathbf{D}}}
\newcommand{\bmU}{\bar{\mathbf{U}}}
\newcommand{\bmW}{\bar{\mathbf{W}}}

\newcommand{\bmC}{\bar{\mathbf{C}}}

\newcommand{\Expect}{\mathbb{E}}
\newcommand{\intd}{\,\mathrm{d}}
\newcommand{\diag}{\mathrm{diag}}

\newcommand{\tr}{\mathrm{tr}}
\newcommand{\Var}{\mathrm{Var}}

\def\trans{^\mathsf{T}}

\usepackage{amsmath}

\DeclareMathOperator*{\argmin}{arg\,min}

\usepackage{accents}

\newcommand{\vertiii}[1]{{\left\vert\kern-0.25ex\left\vert\kern-0.25ex\left\vert #1 
		\right\vert\kern-0.25ex\right\vert\kern-0.25ex\right\vert}}
	\usepackage{comment}

\RequirePackage{amsthm,amsmath,amsfonts,amssymb}
\RequirePackage{natbib}
\RequirePackage{graphicx}
\usepackage{mathrsfs}
\usepackage{amsbsy}
\usepackage{hyperref}
	
\usepackage[normalem]{ulem}

\usepackage{multirow}
\usepackage{tabularx}

\newtheorem{remark}{Remark}

\startlocaldefs

\newtheorem{theorem}{Theorem}[section]
\newtheorem{lemma}[theorem]{Lemma}
\newtheorem{proposition}[theorem]{Proposition}

\theoremstyle{remark}

\newtheorem{condition}{Condition}

\allowdisplaybreaks

\endlocaldefs

\begin{document}

\begin{frontmatter}
	\title{Penalized spline estimation of principal components for sparse functional data: rates of convergence}
	\runtitle{Penalized Spline Estimation for FPCA}
	
	\begin{aug}
		\author[A,B]{\fnms{Shiyuan} \snm{He}\ead[label=e0,mark]{20240101@btbu.edu.cn}}
		\author[C]{\fnms{Jianhua Z.} \snm{Huang}\ead[label=e3,mark]{jhuang@cuhk.edu.cn}}
		\author[B]{\fnms{Kejun} \snm{He}\ead[label=e4,mark]{kejunhe@ruc.edu.cn (corresponding author)}}
		\address[A]{School of Matematics and Statistics, Beijing Technology and Business University, \printead{e0}}
		\address[B]{Center for Applied Statistics, Institute of Statistics and Big Data, Renmin University of China,
			\printead{e4}}
		
		\address[C]{School of Data Science, The Chinese University of Hong Kong, Shenzhen (CUHK-Shenzhen), China,
			\printead{e3}}
	\end{aug}

	\begin{abstract}
		This paper gives a comprehensive treatment of the convergence rates of penalized spline estimators for simultaneously estimating several leading principal component functions, when the functional data is 
		sparsely observed.  The penalized spline estimators are defined as the solution of a penalized empirical risk minimization problem, where the loss function belongs to a general class of loss functions motivated by the matrix Bregman divergence, and the penalty term is the integrated squared derivative. 
		The theory reveals that the asymptotic behavior of penalized spline estimators depends on the interesting interplay between several factors, i.e., the smoothness of the unknown functions, the spline degree, the spline knot number, the penalty order, and the penalty parameter. The theory also classifies the asymptotic behavior into seven scenarios and characterizes whether and how the minimax optimal rates of convergence are achievable in each scenario. 
	\end{abstract}
	
	
	
\end{frontmatter}




\section{Introduction and overview of the main result}\label{sec:intro}

Functional principal component analysis (FPCA, \cite{rice1991estimating,silverman96,ramsay05}) is a widely used dimension reduction tool for analyzing functional data. 
After the principal component reduction, an infinite-dimensional functional data object is summarized by a few leading principal component functions, which in turn can be used as inputs for statistical analysis, such as regression \citep{ramsay91,yao05}, clustering \citep{james03,chiou07}, and classification \citep{james2001functional}. 
There is a large literature on functional principal component analysis and its applications. Recent developments include
\cite{cai2010nonparametric,chiou2014multivariate,happ2018multivariate,Shin2022,Shi2022}.

In many applications, closely related to longitudinal data analysis, functional data objects are only sparsely observed, i.e., each function is only observed at a finite number of points~\citep{rice2004}. 
For such applications, FPCA has another important usage of recovering the whole function by borrowing strength across functions. This is possible because these functional objects share the same set of leading principal component functions. There are two typical approaches for estimating functional principal components of sparse functional data. The first approach estimates the covariance function using a smoothing method such as local polynomial smoothing and then performs eigenvalue decomposition of the estimated covariance kernel (e.g., \cite{yaopca05,hall2006properties,li2010uniform}). 
The second approach models the functional data objects by a linear combination of principal components, which in turn are represented by a basis expansion such as polynomial splines, and then uses maximum likelihood or its variants like (penalized) empirical risk minimization for parameter estimation (e.g., \cite{james2000principal,peng2009geometric,zhou2008joint,Zhou2014}).

The first theoretical work on the asymptotic properties of the second approach has been done in an impressive work of \cite{peng2009geometric}, where the specific manifold structure of the restricted parameter space of spline coefficients for representing principal component functions was thoroughly investigated, and the rates of convergence and asymptotic normality of the spline estimator were obtained. However, this work did not consider using roughness penalty for the spline estimator.  It has been widely accepted that penalized splines are advantageous over un-penalized polynomial splines for function estimation~\citep{eilers96,ruppert2003semiparametric}. Applications of using roughness penalty in spline estimation of functional principal components include \cite{zhou2008joint,Zhou2014,He2018,Ding2022,sang2022}.

Asymptotic properties of the penalized spline function estimation have been studied in a series of works \citep{li2008asymptotics,wang2011asymptotics,claeskens2009asymptotic,kauermann2009some,xiao2019asymptotic,xiao2020asymptotic,huang2021asymptotic}. Our interest lies in extending these works to the problem of penalized spline estimation of functional principal components for sparse functional data.  As we shall describe more precisely in subsequent sections, we represent the leading principal component functions using basis expansions of polynomial splines and estimate the spline coefficients by minimizing a penalized empirical risk (or loss), where the loss function belongs to a general class of matrix divergence losses and the penalty term takes the form of $J(f) = \int \{f^{(q)}\}^2$, where $q$ is referred to as the order of the penalty functional or penalty order for short. 
See equations~\eqref{eqn:penalty}, \eqref{eqn:firstOptimization}, and \eqref{eqn:divergenceloss} for the precise definition of the penalized empirical risk minimization problem.

On asymptotic properties of the penalized spline estimator of functional principal components, the following questions are of particular interest. 
\vspace{-0.4em}
\begin{itemize}[leftmargin=1em]
	\item For the regression problem or more generally for extended linear models, it has been observed in the literature that the penalized spline estimator sometimes behaves like an un-penalized polynomial spline estimator, and sometimes behaves like a smoothing-spline estimator (e.g., \cite{claeskens2009asymptotic, huang2021asymptotic}). Does the penalized spline estimator of principal component functions have similar behaviors? 
	
	\item In \cite{paul2009consistency}, the authors showed that the minimax lower bound \citep{stone1982optimal}
	for estimating a principal component function is $N^{-2p/(2p+1)}$, where $N$ is the number of functional objects and $p$ is the smoothness of the unknown function. If the rate of convergence of an estimator matches this lower bound, we say that the estimator achieves the (minimax) optimal rate of convergence. When can the penalized spline estimator achieve the optimal rate of convergence? How should the number of knots and the penalty parameter be specified in order to achieve the optimal rate?
	
	\item How is the asymptotic behavior of the penalized spline estimator influenced by the penalty order and the number of knots of the splines?
	This question is of important because in practice prior knowledge about the smoothness of the unknown function is usually unavailable. For instance, if one uses a penalty that assumes larger (or less) order of derivatives than what the unknown function actually has, can the penalized spline estimator still achieve the optimal rate?
\end{itemize}
\vspace{-0.4em}
The goal of this paper is to give rather comprehensive answers to all these questions. 

Our conclusion is that the asymptotic behavior of the penalized spline estimator of the principal component functions depends on the interplay among the smoothness $p$ of the unknown function, spline degree $m$, penalty order $q$, spline knot number $K=K_N$, and the penalty parameter $\eta= \eta_N$. 
Note that it is necessary to require $q  \leq m$, otherwise the penalty term is not well-defined (see, e.g., Section~3 of \cite{huang2021asymptotic}).
We also set $\zeta = (m+1)\wedge p$. According to our result presented in Theorem~\ref{thm:main} of Section~\ref{sec:roadmap}, the asymptotic behavior can be classified into seven scenarios summarized in Table~\ref{summary_table}.  
Table~\ref{summary_table} shows that scenarios of convergence rates are characterized by $\eta \lesssim K^{-\nu}$ for some constant $\nu$. Since using a smaller penalty parameter $\eta$ indicates lighter penalization, we refer to these cases as the light penalty scenarios. Alternatively, these scenarios can be equivalently characterized by $K \lesssim \eta^{-1/\nu}$, we may also refer to these cases as the small knot number scenarios. Similarly, the scenarios corresponding to 
$\eta \gtrsim K^{-\nu}$, or equivalently $K\gtrsim \eta^{-1/\nu}$, are referred to as the heavy penalty or large knot number scenarios. 

\begin{table}[t!]
	\caption{Seven scenarios for the rate of convergence of $\|\hat{\psi}_{r} - \psi_{0r} \|^2  + \eta  J (\hat{\psi}_{r})$, where $\hat{\psi}_{r}$ is the penalized spline estimator of the principal component function $\psi_{0r}$. The left column presents the rate of convergence, the middle column gives the specification of parameters for achieving the best rate presented on the right column. In the right column, (*) indicates achieving the minimax optimal rate when $p\leq m+1 $; (**) indicates achieving the minimax optimal rate.}
	\label{summary_table}
	\begin{center}
		\begin{tabular}{c >{\centering\arraybackslash}p{6cm} c}
			\hline
			\textbf{Rate of convergence} &\textbf{Parameters for achieving the best rate} & \textbf{Best rate} \\[0.2em]
			\hline\hline
			\noalign{\vskip 0.05em} 
			\multicolumn{3} {c}{I. $q < p$}\\[0.05em] 
			\hline  \noalign{\vskip 0.05em}    
			\multicolumn{3}{l}{1. $\eta \lesssim K^{-2\zeta}$} \\[0.05em]
			$K^{-2\zeta} + (N/K)^{-1}$
			& $K\asymp N^{1/(2\zeta+1)} $&  $N^{-2\zeta/(2\zeta+1)}$  (*) \\
			\noalign{\vskip 0.05em} 
			\hline  \noalign{\vskip 0.05em} 
			\multicolumn{3} {l}{2. $K^{-2\zeta} \lesssim \eta \lesssim K^{-2q}$}  \\[0.05em]
			$\eta + (N/K)^{-1}$
			& $\eta \asymp K^{-2\zeta}$ and 
			$K\asymp N^{1/(2\zeta+1)}$ &  $N^{-2\zeta/(2\zeta+1)}$  (*)  \\
			\noalign{\vskip 0.05em} 
			\hline  \noalign{\vskip 0.05em} 
			\multicolumn{3} {l}{3. $\eta \gtrsim K^{-2q}$}\\[0.05em]
			$\eta + \{N \eta^{1/(2q)}\}^{-1}$
			& $\eta \asymp N^{-2q/(2q+1)}$ &  $N^{-2q/(2q+1)}$ \\
			\noalign{\vskip 0.05em}
			\hline\hline
			\noalign{\vskip 0.05em}
			\multicolumn{3} {c}{II. $ q= p$ }\\[0.05em] 
			\hline  \noalign{\vskip 0.05em} 
			\multicolumn{3} {l}{1. $\eta \lesssim K^{-2p}$} \\[0.05em]
			$K^{-2p} + (N/K)^{-1}$ 
			& $K \asymp N^{1/(2p+1)}$& $N^{-2p/(2p+1)}$ (**) \\
			\noalign{\vskip 0.05em} 
			\hline  \noalign{\vskip 0.05em} 
			\multicolumn{3} {l}{2. $\eta \gtrsim K^{-2p}$}\\[0.05em]
			$\eta  + \{N \eta^{1/(2p)}\}^{-1}$ 
			& $\eta \asymp N^{-2p/(2p+1)}$ & $N^{-2p/(2p+1)}$ (**) \\
			\noalign{\vskip 0.05em}
			\hline\hline
			\noalign{\vskip 0.05em}
			\multicolumn{3} {c}{III. $ q > p$}\\[0.05em]
			\hline  \noalign{\vskip 0.05em} 
			\multicolumn{3} {l}{1. $\eta \lesssim K^{-2q}$}\\[0.05em]
			$K^{-2p} + (N/K)^{-1}$ 
			& $K \asymp N^{1/(2p+1)}$ & $N^{-2p/(2p+1)}$  (**)\\
			\noalign{\vskip 0.05em} 
			\hline  \noalign{\vskip 0.05em} 
			\multicolumn{3} {l}{2. $\eta \gtrsim K^{-2q}   $}\\[0.05em]
			$ \eta K^{2q-2p} +  \{N \eta^{1/(2q)}\}^{-1}$ 
			&  $K \asymp \eta^{-1/(2q)}$ and $\eta \asymp N^{-2q/(2p +1)}$& $N^{-2p/(2p+1)}$  (**)\\
			\noalign{\vskip 0.05em} 
			\hline
		\end{tabular}
	\end{center}
\end{table}

In the following, we discuss the results in Table~\ref{summary_table} by grouping the seven scenarios according to the relation between the smoothness $p$ of the unknown  function $\psi_{0r}$ and the penalty order $q$.

\begin{enumerate}[leftmargin=2em]
	\item[I.] $q<p$. For the light penalty or small knot number case, Scenario I.1, the penalized spline estimator behaves like an un-penalized polynomial spline estimator. Specifically, the rate of convergence is $K^{-2\zeta} + (N/K)^{-1}$, which is the rate of convergence of an un-penalized polynomial spline estimator of a $p$-smooth function \citep{huang2003local}. If $p \leq m+1$, when the knot number is chosen to satisfy $K\asymp N^{1/(2p+1)}$, the estimator can achieve the optimal rate of convergence $N^{-2p/(2p+1)}$. If $p> m+1$, the best rate can be achieved is $N^{-2(m+1)/\{2(m+1)+1\}}$, which is slower than the optimal rate. This sub-optimal rate is due to the saturation phenomenon of spline approximation (see page~3388 of \cite{huang2021asymptotic}) and thus cannot be improved.
	
	For the heavy penalty or large knot number case, Scenario I.3, the penalized spline estimator behaves like the smoothing spline estimator. Specifically, the rate of convergence is $\eta + \{N \eta^{1/(2q)}\}^{-1}$, which is the rate of convergence of a smoothing spline estimator of a $p$-smooth function \citep{gu2013smoothing}. When the penalty parameter is chosen to satisfy $\eta \asymp N^{-2q/(2q+1)}$, the estimator achieves the best rate  $N^{-2q/(2q+1)}$ for this case, which is slower than the optimal rate. This result suggests that using a penalty with an order smaller than the smoothness of the unknown function will hurt the ability of the penalized spline estimator to achieve the optimal rate.
	
	For the intermediate case, Scenario I.2, the estimator behaves neither like the un-penalized polynomial spline estimator nor like the smoothing spline estimator. The rate of convergence depends on both the penalty parameter $\eta$ and the knot number $K$. If $p\leq m+1$, by appropriately choosing both parameters, the estimator can achieve the optimal rate of convergence. 
	We note, however, that the penalty parameter $\eta$ needs to be chosen on the left boundary of its range for rate optimal. 
	
	\item[II.] $q=p$. Comparing with I.1--I.3, the results are cleaner. There is no intermediate scenario. Scenario II.1 is similar to I.1, with the difference that the optimal rate can always be obtained. Scenario II.2 is similar to I.3, with the difference that the optimal rate of convergence can be obtained. 
	
	\item[III.] $q>p$. Scenario III.1, similar to II.1, is a light penalty or small knot number case. In this case, the penalized spline estimator behaves like an un-penalized polynomial spline estimator and the optimal rate can be achieved by suitably choosing the knot number. Scenario III.2 is a heavy penalty or large knot number case. Unlike in I.3 and II.2, the estimator behaves neither like the un-penalized polynomial spline estimator nor like the smoothing spline estimator. The convergence rate depends on the penalty parameter $\eta$ and the knot number $K$. By appropriately choosing both parameters, the estimator can achieve the optimal convergence rate. We note, however, that the penalty parameter needs to be chosen on the left boundary of its range for rate optimal. 
\end{enumerate}

Similar results have been obtained
for penalized spline function estimation of convex extended linear models in \cite{huang2021asymptotic}; see Table~1 of the cited paper and its discussion. 
We would like to point out the substantial differences of the two papers.
First, the orthonormal constraints on the eigenfunctions impose a manifold structure on the parameter space, and it is necessary to develop an appropriate local manifold geometry on the space of constrained spline coefficients to solve the current problem. Second, the loss function used here is much more complex than those used in the extended linear models, and the corresponding penalized empirical risk minimization problem is non-convex. 
Third, we need new convergence results of an empirical process indexed by a manifold. Because of these differences, we take an entirely new technical approach for penalized splines to obtain our results in Table~\ref{summary_table}.

This paper advances in several important ways the earlier work on rates of convergence of spline estimation of eigenfunctions for sparse functional data~\citep{paul2009consistency}.
The earlier work considered only a loss function motivated by the Gaussian assumption, and the (sub)-Gaussian tail behavior of the observations was assumed.
Our theory, on the other hand, applies to a general class of loss functions motivated by matrix divergence \citep{pitrik2015joint,dhillon2008matrix}, and the (sub)-Gaussian assumption is not required.  
Our theory also characterizes the situations that optimal rates of convergence can be obtained, improving the sub-optimal convergence rate of \cite{paul2009consistency}, which includes an additional logarithmic factor.
More importantly, introducing the roughness penalty allows us to find interesting new phenomena that do not exist for un-penalized spline estimators (as shown in Table~\ref{summary_table} and its discussions).

The empirical performance of the penalized spline estimator of functional principal components presented in this paper has been examined in a recent work \citep{he2022spline}, where a fast manifold conjugate gradient algorithm was developed to solve the optimization problem. The numerical results demonstrated the competitive performance of the penalized spline estimator. The current paper develops the relevant theoretical results that were missing in the previous work.

The rest of this paper is organized as follows. Section~\ref{sec:meth} presents the problem formulation of the penalized spline estimation of functional principal components. 
Section~\ref{sec:lossfun} introduces a general class of loss functions that are used in our formulation and discusses their properties. 
The main theoretical result is stated in Section~\ref{sec:roadmap} and proved in Section~\ref{sec:theo:mainResult}. 
To prepare for the proof in Section~\ref{sec:theo:mainResult},
Section~\ref{sec:manifold} develops the local geometry for the manifold of the constrained parameter space and new convergence results on an empirical process defined on a manifold, and Section~\ref{sec:expectedloss} develops properties of the risk function around the optimal parameter. In Section~\ref{sec:consistency}, we complement the main result by establishing the consistency of the global estimator by restricting the parameter space to be a compact set. Section~\ref{sec:disc} discussed some problems for future research. The Supplementary Material \cite{he2023penalized} contains all the other technical proofs and details. 

\textbf{Notations}. In this work, for two sequences of numbers $\{a_n\}$ and $\{b_n\}$, we write $a_n\lesssim b_n$ if $a_n\le C\cdot b_n$ for some positive constant $C$. When $a_n\lesssim b_n$ and $b_n\lesssim a_n$ simultaneously, their relation is denoted as $a_n \asymp b_n$.
We write $a_n \ll b_n$ if $a_n/b_n \to 0$ as $n\to \infty$. For two numbers $a,b\in \bbR$, we denote $a\wedge b = \min\{a,b\}$ and $a\vee b = \max\{a,b\}$. Denote $\|\cdot\|$ as the operator norm and $\|\cdot\|_F$ the Frobenius norm.  Table~S.1 in the Supplementary Material \cite{he2023penalized} provides a summary of notations and constants that are frequently used in the paper.

\section{Penalized spline estimation for FPCA} \label{sec:meth}
We now present the general framework of using the penalized spline method to estimate several leading functional principal components from sparsely observed functional data. 

\subsection{Functional principal components} \label{sec:meth:defininition}
Consider a random function $x(u)$ with  index $u\in \sU$, where $\sU$ is a compact set in $\mathbb{R}$. Without loss of generality,  we assume $\sU=[0,1]$. Let $ L_2(\sU)$  denote the Hilbert space of squared integrable functions on $\sU$, equipped with the inner product $\langle x, \widetilde x \rangle_{L_2} = \int_{\sU} x(u) \widetilde{x}(u)\intd u$  and the norm  $\|f\|_{L_2}^2 =\int_{\sU} f^2(u)\intd u$, for $x,\widetilde{x}\in L_2(\sU)$.
Define $L_2(\sU^2)$ similarly for $\sU^2 = \sU \times \sU$. Suppose the random function $x(u)$ has zero mean and  covariance function  $\mathcal{K}(u,v ) = \Expect \{x(u) x(v) \}$. The zero mean assumption is simply to focus only on the asymptotics of the estimates of the leading functional principal components. The covariance function is assumed to be sufficiently smooth, as described in the following condition. The derivative below is understood as weak derivative.

\begin{condition} \label{ass:covsmooth}
	The true covariance function $\mathcal{K}$  belongs to the bi-variate Sobolev space of order $p(\ge 1)$, i.e.,
	$$ \mathcal{K} \in L_2^{p}(\sU^2) := \Big\{\mathcal{H}\in L_2(\sU^2):\; \sum_{i+j = p}\Big\|\frac{\partial^p}{\partial u^i\partial v^j} \mathcal{H}(u,v) \Big\|_{L_2} < \infty \Big\}.$$
\end{condition}

Since $\mathcal{K}$ is non-negative definite, it follows from Mercer's theorem (see Section~4.6 of \cite{hsing2015theoretical})  that there exists an orthonormal sequence $\{\psi_{0r}(u)\}_{r=1}^\infty$  of eigenfunctions in $L_2(\sU)$, and a decreasing, non-negative sequence $\{\lambda_{0r}\}_{r=1}^\infty$ of eigenvalues, such that 
\begin{equation}\label{eqn:eigenfunction}
	\int_{\sU} \mathcal{K}(u,v) \psi_{0r}(v) \intd v = \lambda_{0r}\psi_{0r}(u), \qquad r = 1, 2, \ldots, 
\end{equation}
and the covariance function can be expanded as
\begin{equation} \label{eqn:truecov:decompose}
	\mathcal{K}(u, v) = \sum\nolimits_{r=1}^\infty \lambda_{0r}\psi_{0r}(u) 	\psi_{0r}(v)\, .
\end{equation}
By taking weak differential under the integral sign \citep{talvila2001necessary} in the relation \eqref{eqn:eigenfunction}, we have 	
\[
\psi_{0r}^{(k)}(u) =  \frac{1}{\lambda_{0r}} \int_{\sU} \frac{\partial^k }{\partial u^k} 	\mathcal{K}(u, v) \psi_{0r}(v) \intd v
\]
for $k=1,\ldots, p$. Via the Cauchy-Schwarz inequality and the square integrability of $\frac{\partial^k }{\partial u^k} 	\mathcal{K}(u, v) $, we know that $\psi_{0r}^{(k)}(u)$ is also square integrable. 
Therefore, the eigenfunction $\psi_{0r}$ belongs to the Sobolev space of order $p (\ge 1)$, i.e.,
\begin{equation} \label{eqn:eigensmooth}
	\psi_{0r} \in L_2^{p}(\sU) := \big\{\psi:\; \psi^{(k)}\in L_2(\sU), \text{ for } k=1,2,\ldots, p \big\}, \quad r= 1, 2, \ldots
\end{equation}

Consider a fix integer $R\geq 1$. Among the equivalence class (where the function only differs on a zero measure set) of $\mathcal{K}$ and $\psi_{0r}$, we can select $\mathcal{K}$ and $\psi_{0r}$ to be a sufficiently smooth version. Then the smoothness implies a uniform upper bound for the covariance function and the leading $R$ eigenfunctions on their compact domains. Therefore, there exists a constant $C_u$ such that
\begin{equation} \label{eqn:uniformupper}
	\sup_{u,v}\, \mathcal{K}^2(u,v) \le C_u/2 \quad \text{and} \quad \max_{r\le R}\sup_u\,  \psi^2_{0r}(u) \le C_u/2.
\end{equation}
The smoothness of $\mathcal{K}$  implies  the covariance function $\mathcal{K}$ belongs to the trace class, i.e.,
\begin{equation} \label{eqn:define:Clambda}
	C_{\lambda} := \sum\nolimits_{r=1}^{\infty} \lambda_{0r} = \int_{\sU}  \mathcal{K}(u,u) \intd u< \infty;
\end{equation}
the last inequality holds because the covariance function $\mathcal{K}$ is continuous over the compact domain $\sU\times \sU$.

All random functions $x$ with zero mean and covariance function $K(u,v)$ admit  the Karhunen-Lo{\`e}ve expansion
$x(u) = \sum_{r=1}^\infty \lambda_{0r}^{1/2} \theta_r \psi_{0r}(u),$
where $\theta_r$'s are uncorrelated random variables with zero mean and unit variance, and the convergence is in $L_2$ and uniform in $u$. Our interest lies in estimating the leading $R$ eigenfunctions from the data. In practice, the number $R$ is usually chosen such that the $R$ leading functional principal components explain a desired proportion of variability in the dataset. We treat $R$ as a fixed number in this paper. We write informally that
$\mathcal{K}(u, v) \approx \sum_{r=1}^R \lambda_{0r}\psi_{0r}(u) \psi_{0r}(v)$ and $x(u) \approx \sum_{r=1}^R \lambda_{0r}^{1/2} \theta_r \psi_{0r}(u)$.  

To ensure identifiability of the leading $R$ eigenfunctions, we impose a strictly positive lower bound for the relevant eigen gaps. It is a mild condition that has been used in the literature (see, e.g., \cite{hall2006properties, paul2009consistency}). 

\vspace{-0.2em}
\begin{condition}  \label{ass:eigengap}
	The eigen gaps for the leading $R+1$ eigenvalues have a positive lower bound, i.e., 
	$
	\min_{1\le r < r' \le R+1} (\lambda_{0r}
	- \lambda_{0r'}) \ge C_E > 0,$
	for some strictly positive constant $C_E$.
\end{condition}

\subsection{Sparsely observed functional samples}
\label{sec:meth:sparseSample}

In practice, we have $N$ realized functions $x_1, \ldots, x_N$ of $x$, where $x_n$ is sparsely observed at $M_{n}$ time points $u_{n1},\ldots, u_{nM_n}\in\sU$ for $n=1,\ldots, N$. At $u_{nj}$, the observed function value $y_{nj}$ is $x_n(u_{nj})$ corrupted by a noise variable, i.e.$y_{nj} = x_{n}(u_{nj} )+ \epsilon_{nj}$,
for $n=1,\ldots, N$, and $j = 1,2,\ldots, M_{n}$.  
The noise variables $\epsilon_{nj}$'s are independent with zero mean and finite variance $\sigma_e^2 = \Expect \epsilon_{nj}^2$. For the $n$-th function, we define $\vy_n = (y_{n1},\ldots, y_{nM_n})\trans$ and obtain a one-sample covariance matrix $\mS_n = \vy_n\vy_n\trans$. The matrix $\mS_n$ is an unbiased estimate of the true covariance matrix of $\vy_n$, which is
\begin{equation}   \label{eqn:trueCovMat}
	\mK_n = \Var(\vy_n)= \big[\mathcal{K}(u_{nj}, u_{nj'})\big]_{j,j'} + \sigma_e^2\mI.
\end{equation}
Collectively, the  matrices $\mS_1,\ldots, \mS_N$ provide the second moment information about 
the true covariance function $\mathcal{K}$ at the random time points $u_{nj}$'s.
Our goal is to recover the true leading eigenfunctions $\psi_{01},\ldots, \psi_{0R}$ from the sample covariance matrices $\mS_1,\ldots, \mS_N$.

Our theoretical analysis focuses on the situation of sparse functional observations, where the number of observation points $M_n$ has a fixed lower and upper bound. 

\begin{condition} \label{ass:density}
	i)  For some fixed constants $\underline{M}>1$  and $\overline{M} < \infty$, $M_n$ satisfies that $\underline{M}\le M_n \le \overline{M}$ for all $n$. 
	ii) The observation times $u_{nj}$'s are independently sampled according to a  density $G(u)$, whose support is $\sU = [0,1]$. There exist positive constants $c_g$ and $C_g$ such that $0<c_g < G(u) < C_g$, for all $u\in \sU$.
\end{condition}

Part ii) in Condition~\ref{ass:density} means that the probability that $x(u)$ will be observed at $u$ (with noise added) is of the same order of magnitude for all $u \in \mathcal{U}$, i.e., $c_g/C_g < G(u)/G(u') < C_g/c_g$ for all $u, u'\in \mathcal{U}$.

\begin{condition} \label{ass:moment}
	i) The noise variables $\epsilon_{nj}$'s are independent of each other with zero mean and finite variance $\sigma_e^2$. 
	There exists a constant $C_m(\ge 1)$ such that $\Expect \epsilon_{nj}^4\le C_m\sigma_e^4.$\\
\hspace{30pt}	ii) Let $\vx_n = (x_n(u_{n1}), \ldots, x_n(u_{nM_n}))\trans$ be the signal vector and $\vu_n = (u_{n1},\ldots, u_{nM_n})\trans$ be the observational time points. It holds, 	for any fixed vector $\vv\in\bbR^M$, that
	\begin{equation}\label{eqn:ass:moment}
		\Expect \big\{\big(\vx_n\trans \vv\big)^4 | \vu_n\big\}\le C_m
		\Big[\Expect\big\{\big(\vx_n\trans \vv\big)^2 | \vu_n\big\}\Big]^2.
	\end{equation}
\end{condition}

Condition~\ref{ass:moment} is a mild condition that requires the random variables to have finite fourth moment.  
A condition similar to Part ii) was used in~\cite{cai2010nonparametric}, where they bounded the fourth moment of the integral of $\vx_n$ by the square of its second moment. Other forms of the fourth moment assumption were used for the local smoothing methods \citep{hall2006properties, li2010uniform}. A stronger (sub-)Gaussian assumption was used in \cite{paul2009consistency}.

\subsection{An working model based on spline approximations}\label{sec:meth:spline}

Our estimation method relies on an approximate working model that is built upon a class of finite-rank covariance functions defined in a space of tensor product splines. The covariance functions and the corresponding eigenfunctions of the working model are used to approximate the true covariance function and the true eigenfunctions, respectively. 

Consider a $K$-dimensional linear space of spline functions with degree $m \geq 1$ (or equivalently order $m+1$), defined on the domain $\sU$. Let $\vb(\cdot)\in\bbR^K$ be a vector whose elements form an orthonormal basis for this space, i.e., $\int_\sU \vb(u)\vb\trans(u)\intd u = \mI$. Let $\mW\in\bbR^{K\times K}$ be a rank-$R$ positive semi-definite matrix with the eigen-decomposition $\mW = \mU\mD\mU\trans$, where $\mU\in\bbR^{K\times R}$ has orthonormal columns and $\mD=\diag(\lambda_{1},\ldots, \lambda_{R})\in\bbR^{R\times R}$ is a diagonal matrix. This is the same matrix parameterization as adopted in \cite{paul2009consistency, peng2009geometric}. The working model assumes the following rank-$R$ covariance function 
\begin{equation}\label{eq:working-model}
	\CCov(u,v)={\vb}\trans(u)\mW\vb(v)= {\vb}\trans(u)\mU\mD\mU\trans\vb(v),
\end{equation}
which is a tensor product spline. Note that the working model \eqref{eq:working-model} has the parameter $(\mU, \mD)$, which determines $\CCov(u,v)$. Let $\vu_r$ be the $r$-th column of $\mU$ and
let $\psi_r(u) = \vu_r\trans\vb(u)$. It is easy to see that
\[
\int_\sU \CCov(u,v) \psi_r(v) \intd v = \lambda_{r}\psi_{r}(u).
\]
Thus the working model covariance function $\CCov$ has $\psi_r(u)$ as its eigenfunction and $\lambda_{r}$ as its corresponding eigenvalue, $r=1,\ldots, R$. 
We also have that
\begin{equation} \label{eqn:modeleigenfun:delta}
	\int_{\sU} \psi_r(u)\psi_{r'}(u) \intd u = \vu_{r}\trans \Big\{\int_{\sU} \vb(u)\vb\trans(u)\intd u\Big\} \vu_{r'}= \vu_{r}\trans\vu_{r'} = \delta_{rr'},
\end{equation}
where $\delta_{rr'}=1$ if $r=r'$ and $\delta_{rr'}=0$ if $r\neq r'$.
In this formulation, $K$ is allowed to grow with the sample size $N$, and therefore $\CCov(u,v)$ and $\psi_r(u) $ all implicitly depend on $N$.

Using the working model inevitably induces an error due to spline approximation of unknown functions. To quantify this ``approximation error'', we construct the ``optimal parameters'' in the working model that provide the best approximations to the true covariance function and eigenfunctions. 
These best approximations are the direct targets of an estimation method based on the working model. 
To this end, we first define the optimal spline approximations of the leading $R$ true eigenfunctions, and then combine these spline-approximated eigenfunctions with the true eigenvalues to obtain the optimal rank-$R$ covariance function of the working model.

Specifically, let $\bar{\CCov}^*(u,v) = \vb\trans(u)\bmW_K\vb(v)$ be the best tensor product spline approximation of the true covariance function without the rank-$R$ constraint, where 
\begin{equation} \label{eqn:tensorproductOpt}
		\bmW_K:=\argmin_{\mW\in\bbR^{K\times K},\, \mW\trans = \mW} \int_{\sU^2} \big\{\mathcal{K}(u,v) - \vb\trans(u)\mW\vb(v)\big\}^2\intd u\intd v.
\end{equation}
Denote the $r$-th eigenvector (ordered according to eigenvalues) of $\bmW_K$ as $\bar{\vu}_{r}$.
Then the spline function $\bar{\psi}_r(u) = \vb\trans(u)\bar{\vu}_{r}$, the $r$-th eigenfunction of $\bar{\CCov}^*(u,v)$, is regarded as an optimal spline approximation of the true $\psi_{0r}$ in the working model space.

Let $\bmU = \big(\bar{\vu}_{1},\ldots, \bar{\vu}_{R}\big)$.
Further set $\bmD = \diag(\lambda_{01},\ldots, \lambda_{0R})$, a diagonal matrix with the ordered leading $R$ true eigenvalues of $\mathcal{K}$ as its diagonal entries. Let $\bmW = \bmU\bmD\bmU\trans$ and consider
\begin{equation} \label{eqn:oraclerankR}
	\bar{\CCov}(u,v): = \vb\trans(u) \bmW \vb(v) 
	= \sum_{r=1}^{R} \lambda_{0r} 
	\bar{\psi}_r(u) \bar{\psi}_r(v).
\end{equation}
The rank-$R$ tensor product spline function $\bar{\CCov}(u,v)$ is the optimal approximation of the true covariance function $\mathcal{K}(u,v)$ in the working model.

The next two results characterize the properties of $\bar\psi_r$'s and $\bar{\CCov}(u,v)$. The proof can be found in Section~
S.6 of the Supplementary Material \cite{he2023penalized}.

\begin{proposition} \label{prop:approximationerror:eigenfun}
	Set $\zeta = p \wedge(m+1)$. 
	Denote $(x)_+ =x$ if $x>0$ and $(x)_+ =0$ if $x\le 0$. Then, 
	
	(i) It holds that $\,\delta_{K} := \max_{r\le R}\, \|\bar{\psi}_{r} - \psi_{0r}\|_{L_2} =O\big( K^{-\zeta} \big)$.
	
	(ii)	For $q\le m$, there is a constant $C$ such that
	$\,\max_{r\le R} \int_{\sU} \big\{\bar{\psi}^{(q)}_{r}(u)\big\}^2 \intd u \le C K^{2(q-p)_+}.$
\end{proposition}

The error due to rank-$R$ approximation is measured by 
$\omega_{R}^2 :=  \sup_{u,v} \mathcal{K}_{-}^2(u,v)$
where $\mathcal{K}_{-}(u,v) := \mathcal{K}(u,v) - \sum_{r=1}^{R} \lambda_{0r} \psi_{0r}(u)\psi_{0r}(v)$. 
When all the eigenfunctions $\psi_{0r}(u)$ have a uniform upper bound (i.e. $\max_{r } \sup_u \psi_{0r}^2(u)\le C$ by some constant $C$), we can see that  
$\omega_{R} \lesssim \sum_{r=R+1}^{\infty }  \lambda_{0r}$.
We say that the covariance function $\mathcal{K}(u,v)$ is of rank $R$ if $\omega_R =0$.

\begin{proposition}\label{prop:approximationerror:covfun}
	(i) 	With $\delta_{K}$  defined  in Proposition~\ref{prop:approximationerror:eigenfun}, it holds that 
	\begin{align}
		\|\bar{\CCov}(u,v) - \mathcal{K}(u,v) \|_{L_2(\mathcal{U}^2)}
		&= O\big(\delta_{K} +  \omega_R\big),  \label{eqn:covnfun:oracleerro1}\\
		\|\bar{\CCov}(v,v) - \mathcal{K}(v,v) \|_{L_2(\mathcal{U})}
		&= O\big(\delta_{K} +  \omega_R\big).  \label{eqn:covnfun:oracleerro2}
	\end{align}

	(ii) When $K$ is sufficiently large and $\omega_{R}$ sufficiently small, for $C_u$ in~\eqref{eqn:uniformupper}, it holds that
	\begin{equation}
		\max_{r \le R}	  \sup_u \bar{\psi}^2_{r}(u) \le C_u 
		\quad\text{and}\quad	\sup_{u,v} \bar{\CCov}^2(u,v) \le C_u.
		\label{eqn:uniformupper:part2} 
	\end{equation}
\end{proposition}

Note the difference of the two results in Part (i): \eqref{eqn:covnfun:oracleerro1} quantifies the approximation error for the whole function $\bar{\CCov}(u,v)$, while \eqref{eqn:covnfun:oracleerro2} only measures the integrated error on the diagonal (i.e., $u=v$). 

\subsection{Penalized spline estimation}\label{sec:meth:optimization}

Our estimation uses	the penalized spline method which minimizes a criterion function as a summation of two terms. One term measures the model fidelity to the data, and the other term is a roughness penalty that encourages smoothness of estimated eigenfunctions. 

Recall the spline representation of eigenfunctions in Section~\ref{sec:meth:spline}.
The roughness penalty we consider is the integral of the squared $q$-th $(1\le q \le m)$ derivative of the eigenfuctions
\begin{equation} \label{eqn:penalty}
	\mathcal{P}_\eta(\mU) = 
	\eta \sum_{r=1}^{R}  \int_\sU \big\{\psi^{(q)}_{r}(u)\big\}^2
	\intd u = \eta \, \tr\big( \mU\trans \mGamma \mU\big),
\end{equation}
where $\mGamma = \int \vb^{(q)}(u) \{\vb^{(q)}(u)\}\trans \intd u$ is a $K\times K$ matrix solely depending on the basis functions and $\eta$ is the penalty parameter.
Recall that the working model \eqref{eq:working-model} for the covariance function $\CCov(u,v)$ has the parameter $(\mU, \mD)$.
Let $\mathcal{L}(\mU, \mD,\sigma_e^2)$ be a general loss function to be introduced in Section~\ref{sec:lossfun}. We solve the following minimization problem
\begin{equation}\label{eqn:firstOptimization}
	\begin{aligned}
		\min_{\mU, \mD,\sigma_e^2}\, \;& \ell(\mU,\mD,\sigma_e^2) :=\mathcal{L}(\mU,\mD,\sigma_e^2) +\mathcal{P}_\eta(\mU), \\
		\text{subject to }\, & \mU\in \mathrm{St}(R, K),\;
		\mD = \diag(\lambda_{1},\ldots, \lambda_{R})\in\mathbb{D}_+, \text{ and } \sigma_e^2 >0,
	\end{aligned}
\end{equation}
where $\mathrm{St}(R, K)=\{\mU\in\bbR^{K\times R}:\, \mU\trans\mU = \mI\}$ is the Stiefel manifold \citep{edelman1998geometry}, and $\mathbb{D}_+$ is the set of diagonal matrices with positive diagonal elements. Together, $(\mU,\mD)$ can be viewed as a point of the product manifold 
$\sM = \mathrm{St}(R, K) \times \mathbb{D}_+$. 
We let $(\hat{\mU}, \hat{\mD})$ and $\hat{\sigma}^2_e$ denote the optimal solution to problem \eqref{eqn:firstOptimization}, and denote $\hat{\vu}_{r}$ as the $r$-th column of the solution $\hat{\mU}$, $1\leq r \leq R$. The $r$-th eigenfunction is then estimated by $\hat{\psi}_{r}(u) = \vb\trans(u)\hat{\vu}_{r}$.

In the above formulation, for simplicity of presentation, the same tuning parameter $\eta$ is used for estimating all $R$ eigenfunctions. In practice, different tuning parameters can certainly be used on different eigenfunctions, i.e., one replaces the penalty $\mathcal{P}_\eta(\mU)$ in \eqref{eqn:penalty} by
$\sum_{r=1}^{R}  \eta_r \int_\sU \big\{\psi^{(q)}_{r}(u)\big\}^2\,du$.
The asymptotic results developed in this paper remain valid in this more general setting, as long as the tuning parameters have the same order of magnitude, i.e., $\eta_r \asymp \eta$ for some $\eta$ that satisfies the required conditions.

\section{General divergence loss for FPCA}	\label{sec:lossfun}

This section introduces a general class of loss functions which can be used in~\eqref{eqn:firstOptimization} for estimating functional principal components, and discusses 
a few basic properties of the loss functions in this class.
Many popular loss functions are special cases included in this general class; see
Section S.2 of the Supplementary Material \cite{he2023penalized}.

\subsection{The general divergence loss} \label{sec:lossfun:oracle}

The rank-$R$ model covariance function $\CCov(u,v)$ is used to approximate the true covariance function $\mathcal{K}(u,v)$, and their discrepancy should be  measured based on the sparse observations $\{(u_{nj}, y_{nj})\}_{n,j}$. 
For the $n$-th function in the sample, recall from~\eqref{eqn:trueCovMat} that the true covariance matrix of  $\vy_n = (y_{n1},\ldots, y_{nM_n})\trans$ is   $\mK_n = \big[\mathcal{K}(u_{nj}, u_{nj'})\big]_{j,j'} + \sigma_e^2\mI$.
Correspondingly, the model covariance matrix at the observational time points is 
\begin{equation}\label{eqn:modelCovMat}
	\mC_n = \big[
	\CCov(u_{nj}, u_{nj'})\big]_{jj'} + \sigma_e^2\mI,
\end{equation}
which depends on the model parameter $(\mU, \mD)$ through $\CCov(u,v) = \vb(u) \mU\mD\mU\trans \vb\trans(v)$. In particular, when the parameter is fixed at the optimal parameter $(\bmU, \bmD)$, we have the covariance matrix 
\begin{equation}\label{eqn:oracleCovMat}
	\bmC_n = \big[
	\bar{\CCov}(u_{nj}, u_{nj'})\big]_{jj'} + \sigma_e^2\mI,
\end{equation}
where $\bar{\CCov}(u,v): =  \vb\trans(u) \bmU\bmD\bmU\trans \vb(v).$ 

We exploit the matrix Bregman divergence \citep{dhillon2008matrix, pitrik2015joint} to measure the discrepancy between $\mK_n$ and $\mC_n$. 
In particular, let $\varphi(\cdot):\bbR\to \bbR$ be a strictly convex and twice continuously differentiable function. 
To generalize this function on a positive definite matrix $\mC_n \in\bbR^{M_n\times M_n}$, consider eigendecomposition $\mC_n = \mF_n\mG_n\mF_n\trans$, where the columns of $\mF_n$ contain the eigenvectors of $\mC_n$ and $\mG_n = (g_{n1},\ldots, g_{nM_n})$ is a diagonal matrix of its eigenvalues. 
We denote $\varphi(\mC_n): = \mF_n\varphi(\mG_n)\mF_n\trans$, where $\varphi(\mG_n) = \diag\{\varphi(g_{n1}),\ldots, \varphi(g_{nM_n})\}$ such that $\varphi$ is applied elementwisely to the diagonal elements. 
The function $\varphi$ induces a matrix Bregman divergence for positive definite matrices as 
\begin{equation} \label{eqn:define:divergence}
	\mathcal{D}_{\varphi}(\mK_n || \mC_n) = \tr\big\{\varphi(\mK_n) - \varphi(\mC_n) - \varphi'(\mC_n)(\mK_n - \mC_n)\big\},
\end{equation}
and $\varphi$ is also named the \textit{seed function}.  In particular,
this divergence can be used to gauge the similarity between the true covariance matrix $\mK_n$ and our model covariance matrix $\mC_n$. The divergence is zero if and only if $\mK_n=\mC_n$ (see Lemma~1 of \cite{kulis2009low}).

In practice, the true covariance matrix $\mK_n$ is not observed. We make use of its one-sample estimate $\mS_n = \vy_n\vy_n\trans$, which is an unbiased estimate of $\mK_n$.
Furthermore, notice on the right hand side of~\eqref{eqn:define:divergence}, the first term $\varphi(\mK_n)$ is fixed and does not depend on our model. 
Based on these observations, we remove the first term on the right hand side of~\eqref{eqn:define:divergence}, and replace $\mK_n$ with $\mS_n$ in the last term $ \varphi'(\mC_n)(\mK_n - \mC_n)$.
This leads to the general loss function
\begin{equation} \label{eqn:divergenceloss}
	\mathcal{L} (\mU,\mD,\sigma_e^2) = \frac{1}{N}\sum_{n=1}^{N} \frac{1}{M_n^2}\tr
	\big\{- \varphi(\mC_n) -  \varphi'(\mC_n)(\mS_n - \mC_n)\big\},
\end{equation}
where the matrices $\mC_n$'s on the right hand side depend on the parameters $\mU$ and $\mD$ implicitly through~\eqref{eqn:modelCovMat} with $\CCov(u,v)=\vb\trans(u)\mU\mD\mU\trans\vb(v)$.

Taking expectation conditional on all the observational points $\{u_{nj}\}$, we get the frequentist expected loss function (or called the risk function)
\begin{equation} \label{eqn:divergenceloss:population}
	\mathcal{L}_\infty (\mU,\mD,\sigma_e^2) :=\Expect\big[	\mathcal{L} (\mU,\mD)  | \{u_{nj}\}  \big]
	= \frac{1}{N} \sum_{n=1}^N  \frac{1}{M_n^2}	\mathcal{D}_{\varphi}(\mK_n || \mC_n)  + \mathrm{Const.},
\end{equation}
where $\mathrm{Const.}= \sum_{n=1}^N \tr\{-\varphi(\mK_n)\}$ is a constant that does not depend on the parameter $(\mU, \mD)$. The expected loss function is minimized when $\mC_n=\mK_n$ for all $n$, or when $\CCov(\cdot,\cdot)\equiv \mathcal{K}(\cdot,\cdot)$. 

\subsection{Choice of seed function}	 \label{sec:lossfun:seed}

We require  the first-order derivative  $\varphi'$ of the seed function to be matrix monotone. More precisely, the function $\varphi'$ is called \textit{matrix monotone}, if for two positive definite matrices $\mA,\mB$ of the same size, $\mA\succeq \mB$ implies $\varphi'(\mA)\succeq \varphi'(\mB)$. The  collection of seed functions with matrix monotone first-order derivatives gives a class of useful matrix divergences. For example, the divergence with $\varphi(x) = x\log(x) - x$ is the \textit{von Neumann divergence}, the divergence with $\varphi(x) = -\log(x)$ is the \textit{LogDet divergence}, and  $\varphi(x) = x^2$ leads to the \textit{Frobenius norm loss}~\citep{kulis2009low}. 
These correspond to some popular choices of loss function in the literature of functional principal component analysis \citep{cai2010nonparametric, paul2009consistency}. See Section~
S.2 of the Supplementary Material \cite{he2023penalized} for details.

A matrix monotone function $\varphi'$ enjoys a set of appealing properties. In particular, it is stable with respect to perturbation (see Lemma~
S.3.1 in the Supplementary Material \cite{he2023penalized}). More importantly, an arbitrary matrix monotone function $\varphi'$ has the following general expression
\begin{align}\label{eqn:monotonephi}
	\varphi'(x) =  \alpha+ \beta x + 
	\int_{0}^{\infty}\Big(\frac{\xi}{\xi^2+1} - \frac{1}{\xi+x} \Big)\intd\mu(\xi),
\end{align}
for some $\alpha\in\bbR$, $\beta \ge 0$, and $\mu$ is a non-negative measure satisfying 
\begin{equation}\label{eqn:monotoneCmu}
	C_{\mu} = \int_0^\infty \frac{1}{\xi^2+1} \intd \mu(\xi) < \infty.
\end{equation}
See, for example, Eqn~(V.49) and Eqn~(V.50) in \cite{bhatia2013matrix}. 
We require $\beta + C_\mu >0$ to ensure strict convexity of $\varphi$, as in this case it holds that
$$
\varphi''(x) =   \beta + 
\int_{0}^{\infty}  \frac{1}{(\xi+x)^2} \intd\mu(\xi) \ge \beta + 
\int_{0}^{\infty}  \frac{1}{2(\xi^2+x^2)} \intd\mu(\xi) \ge
\beta +\frac{ C_{\mu} }{ 2(x^2 \vee 1) } > 0.
$$
More examples and  properties of matrix monotone functions can be found in Chapter~V of~\cite{bhatia2013matrix} and Chapter~4 of~\cite{hiai2014introduction}.

\subsection{Local strong convexity of the matrix Bregman divergence}

Under the matrix monotone requirement on $\varphi'$,  local strong convexity of  $\mathcal{D}_{\varphi}(\mK_n || \mC_n)$ in~\eqref{eqn:divergenceloss:population} can be established. According to  \cite{pitrik2015joint} (see Lemma~
S.4.1 and Equation~
(S.1) in Section~
S.4 of the Supplementary Material \cite{he2023penalized}), 
$\mathcal{D}_{\varphi}(\mK_n || \mC_n) $ can be expressed as an integral 
\begin{align} \label{loss:intExp}
	\mathcal{D}_{\varphi}(\mK_n || \mC_n) =\int_0^1 s \times \tr \Big\{ (\mC_n-\mK_n)
	D\varphi'\big(\widetilde{\mC}_n(s) \big)[\mC_n-\mK_n]\Big\} \intd s,
\end{align}
where $\widetilde{\mC}_n(s) = \mK_n+s(\mC_n-\mK_n)$,  $s\in[0,1]$, is a linear interpolation of $\mK_n$ and $\mC_n$ in the corresponding matrix space, and
$D\varphi'(\mA)[\mH] := \frac{\intd}{\intd t} \varphi'(\mA+t\mH)\big|_{t=0}$
is the \textit{directional derivative} of $\varphi'$ at $\mA$ in the direction of $\mH$. 
Using the general expression~\eqref{eqn:monotonephi} of the matrix monotone function $\varphi'$, we have that
\begin{align} \label{eqn:matrixmonotone:derivative}
	D\varphi'(\mA)[\mH] = \beta \mH + \int_0^\infty (\xi\mI + \mA)^{-1} \mH (\xi\mI + \mA)^{-1} \intd \mu(\xi).
\end{align}
Denote $\widetilde{\mDelta}_n = \mK_n^{-1/2}(\mC_n-\mK_n)\mK_n^{-1/2}$ 
and $W_n(\xi,s) = (\xi\mK^{-1}_n + \mI+ s\widetilde{\mDelta}_n)^{-1/2}$. 
Plugging these expressions into~\eqref{loss:intExp}, we obtain
\begin{equation}\label{eq:matBregman-divergence}
	\mathcal{D}_{\varphi}(\mK_n || \mC_n)  = \frac{\beta}{2} \| \mC_n-\mK_n\|_F^2+
	\int_0^1 \int_0^\infty s \times \| W_n(\xi,s)\widetilde{\mDelta}_n W_n(\xi,s)\|_F^2\intd \mu(\xi)  \intd s.
\end{equation}
On the right hand side of \eqref{eq:matBregman-divergence}, the first term is simply the Frobenius norm of the difference between $\mC_n$ and $\mK_n$, while the second term is a mixture of losses scaled by the weight matrix $W_n(\xi,s)$.  

Lemma~\ref{lemma:basicconvexity} below shows that,
when $\mC_n$ and $\mK_n$ are close to each other, $\mathcal{D}_{\varphi}(\mK_n || \mC_n)$ is strongly convex. The proof is given in Section~
S.4 of the Supplementary Material \cite{he2023penalized}.
\vspace{-0.5em}
\begin{lemma} \label{lemma:basicconvexity}
	When $\|\mK_n^{-1/2}(\mC_n-\mK_n)\mK_n^{-1/2}\|<1/2$, each divergence term $\mathcal{D}_{\varphi}(\mK_n || \mC_n)$ is strongly convex such that
	$\mathcal{D}_{\varphi}(\mK_n || \mC_n) \ge 
	\frac{\beta}{2} \| \mC_n-\mK_n\|_F^2+
	\frac{C_\mu}{5}   \| \mK_n^{-1/2}(\mC_n-\mK_n)\mK_n^{-1/2}\|_F^2,$
	where $\beta$ and $C_\mu$ are constants defined in \eqref{eqn:monotonephi} and \eqref{eqn:monotoneCmu}, respectively.
\end{lemma} 

\section{Main result: convergence rates}
\label{sec:roadmap}

This section presents the main result of this paper: the convergence rates of the penalized spline estimator obtained by solving the minimization problem~\eqref{eqn:firstOptimization} with the general divergence loss defined in~\eqref{eqn:divergenceloss}. Recall that $\zeta = p \wedge(m+1)$.  

\begin{theorem} \label{thm:main}
	Assume Conditions $\ref{ass:covsmooth}$--$\ref{ass:moment}$ hold and that
	the covariance function $\mathcal{K}(u,v)$ is of rank $R$ (i.e. $\omega_R=0$).
	Assume $N, K\to\infty$,
	$K^2\log(K)/N\to 0$, 
	and $\eta K^{1 + 2(q-p)_+}\to 0$.
	Then, there exists a local estimator $(\hat{\mU},\hat{\mD}, \hat{\sigma}_e^2)$ of~\eqref{eqn:firstOptimization} such that for $1\leq r \leq R$, 
	\begin{equation}\label{eqn:main:generalresult}
		\|\hat{\psi}_{r} - \psi_{0r} \|^2  + \eta  J (\hat{\psi}_{r} )  \\
		= O_p\Big(\frac{1}{N \eta^{1/(2q)}} \wedge \frac{K}{N}+ \eta K^{2 (q-\zeta)_+}
		\vee \frac{1}{K^{2\zeta}}\Big),
	\end{equation}
	where $\hat{\psi}_{r}(u) = \vb\trans(u)\hat{\vu}_{r}$ for $\hat{\vu}_r$ being the $r$-th column of $\hat{\mU}$.
\end{theorem}

Theorem~\ref{thm:main} directly leads to the rates of convergence in Table~\ref{summary_table}.
Note that, by giving upper bounds on $\|\hat{\psi}_{r} - \psi_{0r} \|^2  + \eta  J (\hat{\psi}_{r})$ instead of  just $\|\hat{\psi}_{r} - \psi_{0r} \|^2$, this result also tells us about the smoothness property of the estimators through the upper bound of the penalty functional $J (\hat{\psi}_{r})$.  The rate of convergence presented in Theorem~\ref{thm:main} is comparable to a recent result obtained in \cite{huang2021asymptotic} in the context of non-parametric regression using penalized splines. On the right hand side of \eqref{eqn:main:generalresult}, the first term corresponds to the estimation error, and the second term corresponds to approximation error and penalty bias. See Section 3 of \cite{huang2021asymptotic} for a comprehensive discussion of the relation of its result with earlier works on rates of convergence of penalized splines, such as \cite{claeskens2009asymptotic,xiao2019asymptotic}.

The condition $K^{2} \log (K) / N \rightarrow 0$ is stronger than the usual condition 
$K/N \rightarrow 0$ or $K\log K/N \rightarrow 0$ used in non-parametric regression \citep{claeskens2009asymptotic,xiao2019asymptotic}. 
We think the $K^2$ cannot be relaxed to $K$ in our condition, because tensor product splines are used to fit the two-dimensional covariance function.

In this article, we will present the proof with a fixed and known $\sigma_e^2$. 
When  $\sigma_e^2$ is unknown,  we only need to augment the parameter space with an additional dimension, and the technical proof  is entirely analogous. Since $\sigma_e^2$ is a scalar parameter, an analysis together with $\sigma_e^2$ does not affect the non-parametric rate of convergence for FPC estimators.  We will thus write $\ell(\mU,\mD,\sigma_e^2)$ as $\ell(\mU,\mD)$, and write $\mathcal{L}(\mU,\mD,\sigma_e^2)$ as $\mathcal{L}(\mU,\mD)$. For simplicity of presentation and  without loss of generality, we also fix $\sigma_e^2=1$ as in the work of \cite{paul2009consistency}.   We now outline the main idea of the proof of Theorem~\ref{thm:main}. The complete proof is given in Section~\ref{sec:theo:mainResult}, based on the technical tools developed in Sections~\ref{sec:manifold} and \ref{sec:expectedloss}.  

Recall the definition of the general divergence loss function
$\mathcal{L}(\mU,\mD)$ given in~\eqref{eqn:divergenceloss} and the corresponding risk function $\mathcal{L}_\infty(\mU,\mD)$ given in~\eqref{eqn:divergenceloss:population}.
The objective function in the minimization problem~\eqref{eqn:firstOptimization},
$\ell(\mU, \mD) = \mathcal{L}(\mU,\mD) +\mathcal{P}_\eta(\mU)$,
is called the penalized loss function (or penalized empirical risk function), while its expectation (conditional on the observational time points), $\ell_\infty(\mU, \mD) = \mathcal{L}_\infty(\mU,\mD) +\mathcal{P}_\eta(\mU)$,
is called the penalized risk function.

Consider the optimal model parameter $(\bmU,\bmD)$ in the approximate working model, as defined in Section~\ref{sec:meth:spline}. We show that there exists a local minimizer of the problem~\eqref{eqn:firstOptimization} inside a suitably defined manifold geodesic neighborhood of $(\bmU,\bmD)$; the radius of the neighborhood measured using a suitable metric 
is roughly the rate of convergence of the estimator.  This amounts to analyzing the difference of the penalized loss $\ell(\mU,\mD)-\ell(\bmU,\bmD)$ locally around  $(\bmU,\bmD)$.

The difference of the penalized loss function at the parameter $(\mU, \mD)$ and at the optimal parameter $(\bmU, \bmD)$ can be written as
\begin{align} 
	\ell(\mU,\mD)-\ell(\bmU,\bmD) & = 
	\big\{	\mathcal{L}(\mU,\mD) +\mathcal{P}_\eta(\mU) \big\}- 
	\big\{	\mathcal{L}(\bar\mU, \bar\mD) +\mathcal{P}_\eta(\bmU)\big\}   \nonumber\\
	& =   \underbrace{\big\{	\mathcal{L}_\infty(\mU,\mD) +\mathcal{P}_\eta(\mU) \big\} }_{\ell_\infty(\mU,\mD)}- \ 
	\underbrace{\big\{	\mathcal{L}_\infty(\bmU,\bmD) +\mathcal{P}_\eta(\bmU)\big\} }_{\ell_\infty(\bmU,\bmD)} \nonumber \\
	&\qquad   + \underbrace{\mathcal{L} (\mU,\mD) -	\mathcal{L}_\infty (\mU,\mD)  - \big\{\mathcal{L} (\bmU,\bmD) -\mathcal{L}_\infty (\bmU,\bmD)\big\} }_{	\mathcal{G}(\mU,\mD) }.  \label{eqn:lossdiff:base}	
\end{align}
Taking the difference of \eqref{eqn:divergenceloss} and \eqref{eqn:divergenceloss:population}, we obtain
\begin{equation}\label{eqn:divergenceloss:centered}
	\mathcal{L} (\mU,\mD) -	\mathcal{L}_\infty (\mU,\mD) 
	=\frac{1}{N}\sum_{n=1}^N \frac{1}{M_n^2} \langle -\varphi'(\mC_n), \,\mS_n - \mK_n\rangle.  
\end{equation}
Using \eqref{eqn:divergenceloss:centered}, the last line of~\eqref{eqn:lossdiff:base} can be written as
\begin{equation} \label{eqn:define:Eprocess}
	\begin{split}
		\mathcal{G}(\mU,\mD)  & :=  \mathcal{L} (\mU,\mD) -	\mathcal{L}_\infty (\mU,\mD)  -\big\{\mathcal{L} (\bmU,\bmD) -\mathcal{L}_\infty (\bmU,\bmD) \big\}\\
		&=  \frac{1}{N}\sum_{n=1}^N \frac{1}{M_n^2} \langle\varphi'(\bmC_n)-\varphi'(\mC_n), \,\mS_n - \mK_n\rangle,
	\end{split}
\end{equation}
which can be viewed as an empirical process indexed by $(\mU,\mD)$.

Based on the decomposition~\eqref{eqn:lossdiff:base}, the remaining of this article will take several steps to complete the proof. First, in Section~\ref{sec:manifold:intro}--\ref{sec:manifold:localTangent}, we develop the product manifold geometry  for the parameter space of the working model. 
Second, in Section~\ref{sec:manifold:EP},
we develop a bound in probability of the supreme of the empirical process $\mathcal{G}(\mU,\mD)$ for $(\mU,\mD)$ in a local geodesic neighborhood of $(\bmU,\bmD)$. Third, in Section~\ref{sec:expectedloss}, we show that 
the difference of the penalized risk, $	\ell_\infty(\mU,\mD) -	\ell_\infty(\bmU,\bmD)$, is lower bounded by a quadratic function of 
the size of the manifold tangent vector. Finally, in
Section~\ref*{sec:theo:mainResult}, we suitably define the size of a geodesic neighborhood of 
$(\bmU,\bmD)$ so that $\ell(\mU,\mD)>\ell(\bmU,\bmD)$ for all boundary points in this neighborhood, and complete the proof.

\vspace{-0.2em}
\begin{remark}~\label{remark:moving-truth}
	The assumption that the covariance function is of finite rank can be relaxed by considering a sequence of true covariance functions that vary with the sample size $N$ such that the rank-$R$ approximation error converges to zero, i.e., $\omega_{R}K^{1/2} \to 0$ as $N \to \infty$.  Under this ``moving truth'' asymptotic setup, the same proof of Theorem~\ref{thm:main} yields the following revision of \eqref{eqn:main:generalresult}:
	\begin{equation}\label{eqn:main:generalresult-2}
		\|\hat{\psi}_{r} - \psi_{0r} \|^2  + \eta  J (\hat{\psi}_{r} )  \\
		= O_p\bigg(\frac{1}{N \eta^{1/(2q)}} \wedge \frac{K}{N}+ \eta K^{2 (q-\zeta)_+}
		\vee \frac{1}{K^{2\zeta}} + \omega_R^2\bigg).
	\end{equation}
	The idea of considering a sequence model was suggested in the work of \cite[][page~1236]{paul2009consistency} where the asymptotics for the un-penalized spline estimation of functional principal components was studied.
\end{remark}

\section{Local manifold geometry and empirical process}\label{sec:manifold}

The main result of this section (presented in Section~\ref{sec:manifold:EP}) controls the supremum of the empirical process  $G(\mU,\mD)$ (defined in~\eqref{eqn:define:Eprocess}) for $(\mU,\mD)$ in a local neighborhood of $(\bmU,\bmD)$ on the manifold $\sM$. To prepare for establishing the main result, Section~\ref{sec:manifold:intro} reviews a few concepts and existing results on local manifold geometry. 
Section~\ref{sec:manifold:localTangent} develops some novel results on the local geometry by using a new metric on the manifold tangent space that is suitable for studying the convergence properties of penalized spline estimators. 
Section~\ref{sec:manifold:EP} characterizes the complexity of the local neighborhood of $(\bmU,\bmD)$ and the  associated empirical process.

\subsection{Local manifold geometry}\label{sec:manifold:intro}

Recall that the model parameter $\mU$ belongs to the Stiefel manifold  $\mathrm{St}(R, K)$, and $\mD$ belongs to the set $\mathbb{D}_+$ of positive diagonal matrices.  Together, $(\mU, \mD)$ is viewed as a point on the product manifold $\sM = \mathrm{St}(R, K) \times \mathbb{D}_+$. Our convergence analysis is based on the local manifold geometry around the optimal parameter of the working model, $(\bmU,\bmD)\in\sM$ (defined in Section~\ref{sec:meth:spline}), which  is constructed assuming the true covariance function $\mathcal{K}$ is known. 

A general study of geometry of the Stiefel manifold can be found in \cite{edelman1998geometry}, and the geometry of positive definite matrices in \citep{bhatia2009positive} is used for $\mD$. The rest of this subsection summarizes some relevant results on the manifold geometry developed in \cite{paul2009consistency} and \cite{chen2012sparse}.

For the Stiefel manifold $\text{St}(R, K)$, its tangent space $\sT_{\bmU} \text{St}(R, K)$ at $\bmU$ consists of all matrices of the form $\mDelta_{u1} = \bmU \mG_u + \mH_u$, where  $\mG_u\in\bbR^{R\times R}$ is a skew-symmetric matrix and $\mH_u\in\bbR^{K\times R}$ is a matrix orthogonal to $\bmU$, i.e., $\mH_u\trans\bmU = \vzero$.  It is easy to check $\|\mDelta_{u1}\|_F^2 = \| \mG_u\|_F^2 +   \| \mH_u\|_F^2$. 
The tangent space $\sT_{\bmU} \text{St}(R, K)$  can be viewed as a first-order approximation to the Stiefel manifold in the ambient space $\bbR^{K\times R}$.  

Given a tangent vector $\mDelta_{u1}$, the geodesic $\exp_{\bmU}(t, \mDelta_{u1})$ is a curve over the Stiefel manifold for $t\in I$, where $I$ is an interval containing $0$. The velocity of the  geodesic at $t=0$ is $\mDelta_{u1}$, and the velocity has zero acceleration for all $t\in I$.   The geodesic  of the  Stiefel manifold has the explicit expression
\begin{equation} \label{eqn:stiefelExpMap}
	\exp_{\bmU}(t, \mDelta_{u1})  =\begin{pmatrix}
		\bmU & \mQ
	\end{pmatrix} \times \exp\big( t\mS \big) \times \begin{pmatrix}
		\mI_R \\ \vzero
	\end{pmatrix},
	\quad\text{where}\quad
	\mS = \begin{pmatrix}
		\mG_u & -\mR\trans \\
		\mR & \vzero
	\end{pmatrix}.
\end{equation}
In \eqref{eqn:stiefelExpMap}, $\exp\big( t\mS \big)$ is the matrix exponential of $t\mS$, and the two matrices $\mQ$ and $\mR$ are respectively the Q and R factors of the QR decomposition of $\mH_u$. For any orthonormal matrix $\mU$ in a small neighborhood of $\bmU$, we can find $\mDelta_{u1}$ in the tangent space of the Stiefel manifold at $\bmU$ such that
\begin{equation} \label{eqn:stiefelLocal}
	\mU = \exp_{\bmU}(1, \mDelta_{u1}) = \bmU+\mDelta_{u1} + O(\|\mDelta_{u1}\|_F^2).
\end{equation}
The geodesic $\exp_{\bmU}(1, \mDelta_{u1})$ with $t=1$ is called the exponential mapping, which maps a tangent vector $\mDelta_{u1}$ to  $\exp_{\bmU}(1, \mDelta_{u1})$ on the manifold.  On the right hand side of~\eqref{eqn:stiefelLocal}, $\mDelta_{u1}$ can be viewed as the first order approximation of the local difference $\mU - \bmU$.

The tangent space for $\mathbb{D}_+$ at  $\bmD$ consists of all matrices 
$\mDelta = \bmD\mDelta_{d1}$, which are parameterized by a diagonal matrix  $\mDelta_{d1}$  (not necessarily with positive diagonal elements). 
The canonical geometry of manifold $\mathbb{D}_+$ is endowed the intrinsic metric $ \| \mD^{-1} \mDelta\|_F= \|\mDelta_{d1}\|_F$.
Via the geodesic
$\exp_{\bmD}(t,\mDelta)= \bmD \exp( \bmD^{-1}\mDelta\cdot t) $, we have that any $\mD$  in a local neighborhood of $\bmD$ can be expressed as
\begin{equation} \label{eqn:psdLocal}
	\mD = \exp_{\bmD}(1, \bmD\mDelta_{d1}) = \bmD +\bmD\mDelta_{d1} + O(\|\mDelta_{d1}\|_F^2),
\end{equation}
by a unique tangent vector $\mDelta_{d1}$. 

Recall that $\bmW = \bmU\bmD\bmU\trans$, which is the coefficient matrix of the optimal covariance function $\bar{\CCov}(u,v)$ in the working model space. Based on~\eqref{eqn:stiefelLocal} and~\eqref{eqn:psdLocal}, the local structure of a rank-$R$ matrix $\mW$ around $\bmW$ can be investigated. 
In particular, 
plugging ~\eqref{eqn:stiefelLocal}  and~\eqref{eqn:psdLocal} into the difference of $\mW$ and $\bmW$, we get
\begin{equation}\label{eqn:local:lowrank1}
	\begin{aligned}
		\mW- \bmW &= \mU\mD\mU\trans -  \bmU\bmD\bmU\trans = \mDelta_{w1} + \mDelta_{w2}, \\ 
		\mDelta_{w1} &:=\bmU\bmD\mDelta_{u1}\trans +\mDelta_{u1}\bmD\bmU\trans + 
		\bmU\bmD\mDelta_{d1} \bmU\trans, \\
		\mDelta_{w2} &= O\big(\|\mDelta_{u1}\|_F^2 + \|\mDelta_{d1}\|_F^2\big).
	\end{aligned}
\end{equation}
In the above, $\mDelta_{w1}$ is interpreted as the first order approximation of the difference $\mW- \bmW$; and  $\mDelta_{w2}$ is the remaining higher-order discrepancy.

As a direct implication of Lemma~3 and Lemma~4 of \cite{chen2012sparse}, the size of the difference $\| \mW - \bmW \|_F^2$  can be approximately measured  by the size of the tangent vectors $\mDelta_{u1}$ and $\mDelta_{d1}$. Their results are summarized in the lemma below. 

\vspace{-0.5em}
\begin{lemma} \label{lemma:local:frobequiv}
	Let $C_E$ be the lower bound of eigen gaps in Condition~\ref{ass:eigengap}. A  lower bound of $\| \mW - \bmW \|_F^2$ takes the form of
	\begin{equation} \label{eqn:localfrobnorm:upper}
		\| \mW- \bmW \|_F^2 \ge C_E^2 \big(\| \mDelta_{u1}\|_F^2+ \|\mDelta_{d1}\|_F^2\big) + O(\|\mDelta_{w1}\|_F^4).
	\end{equation}
	Further, an upper bound of $\| \mW - \bmW \|_F^2$ holds that
	\begin{equation} \label{eqn:localfrobnorm:lower}
		\| \mW- \bmW \|_F^2 \le 2\lambda_{01}^2 \big( \| \mDelta_{u1}\|_F^2 + \|\mDelta_{d1}\|_F^2\big) + O(\|\mDelta_{w1}\|_F^4),
	\end{equation}
	where $\lambda_{0r}$ is the largest eigenvalue of the true covariance function $\mathcal{K}(u, v)$ in~\eqref{eqn:eigenfunction}.
\end{lemma}
\vspace{-0.5em}

According to~\eqref{eqn:localfrobnorm:upper} and~\eqref{eqn:localfrobnorm:lower}, the norm $\| \mW- \bmW \|_F$ is locally equivalent to $\| \mDelta_{u1}\|_F + \|\mDelta_{d1}\|_F$. 
The metric $\|\mDelta_{u1}\|_F + \|\mDelta_{d1}\|_F$ induces a sphere-like neighborhood of the optimal parameter $(\bar\mU, \bar\mD)$ on the tangent space of the product manifold 
$\sM = \mathrm{St}(R, K) \times \mathbb{D}_+$.
Though the metric $\|\mDelta_{u1}\|_F + \|\mDelta_{d1}\|_F$ is useful in the setup of \cite{paul2009consistency} and \cite{chen2012sparse},  it does not appropriately characterize the role of the roughness penalty in our penalized spline setting. 
To study the convergence properties of penalized spline estimators, we need another metric that induces an ellipsoid-like local neighborhood on the tangent space. 

\subsection{Local geometry on the tangent space}\label{sec:manifold:localTangent}

Recall $L_2^q(\sU)$ represent the collection of $L_2$ functions on $\sU$ whose $q$-th order derivative is squared integrable. Consider two quadratic functionals $V$ and $J$ defined for $f\in L_2^{q}(\sU)$,
\begin{equation} \label{eqn:twoQudraticFormsGen}
	V(f) := \int_{\sU} f^2(u)\intd u \quad \text{ and } \quad 
	J(f) : = \int_{\sU}\big\{f^{(q)}(u)\big\}^2\intd u .
\end{equation}
Convergence rates of smoothing spline and penalized spline estimators have been given in the form of $V(\hat{f}-f) + \eta J(\hat{f}-f)$ (see, e.g., \cite{gu2013smoothing, huang2021asymptotic}), where $\hat{f}$ is an estimator of $f$ and $\eta$ is the penalty parameter. This suggests that $V(f) +\eta J(f)$ can be used to introduce a new metric for our purpose. We shall present the definition of the new metric in the coefficient space of a suitable basis expansion of spline functions. 

In particular, when $f(\cdot)$ is expressed by a B-spline basis  $\vb(\cdot)$, i.e., $ f(\cdot) =\vbeta\trans \vb(\cdot)$,  we can equivalently express \eqref{eqn:twoQudraticFormsGen} as
\begin{equation} \label{eqn:twoQudraticForms}
	V(f) =  \vbeta\trans\mN \vbeta \quad \text{ and } \quad 
	J(f)  = \vbeta\trans\mGamma\vbeta,
\end{equation}
where $\mN = \int_{\sU}\vb(u)\{\vb(u)\}\trans\intd u$ and $\mGamma = \int_{\sU}\vb^{(q)}(u) \{\vb^{(q)}(u)\}\trans\intd u$ are two square matrices of size $K$. 
The proof of the following simultaneous diagonalization result can be found in Section~S.5 of the Supplementary Material \cite{he2023penalized}.

\begin{lemma} \label{lemma:splinediagonal}
	There exists a basis $\vb(\cdot)$ such that
	$	\mN = \int_{\sU}\vb(u)\vb\trans(u)\intd u = \mI$ is the identity matrix and
	$\mGamma = \int_{\sU}\vb^{(q)}(u)\{\vb^{(q)}(u)\}\trans\intd u$
	is a  diagonal matrix with elements
	\begin{equation*} 
		\gamma_1=\ldots = \gamma_q = 0, \ \text{ and }\  
		\gamma_j \gtrsim j^{2q},\ \text{ for } j=q+1,\ldots, K.
	\end{equation*}
\end{lemma}

Using the basis in Lemma~\ref{lemma:splinediagonal} that simultaneously diagonalizes $\mN$ and $\mGamma$, we have
\[
V(f) + \eta J(f) = \sum_{k=1}^{K} (1+\eta \gamma_k) \beta_k^2 = \|\vbeta\|^2 + \eta \vbeta\trans\mGamma\vbeta,
\]
which can be treated as a squared norm on the coefficient vector $\vbeta = (\beta_1, \ldots, \beta_K)\trans$.
Extending this norm to a general matrix $\mA\in\bbR^{K\times C}$ with arbitrary number $C$ $(\ge 1)$ of columns, we define 
\begin{equation} \label{eqn:penaltynorm}
	\|\mA\|_{\eta}: = \big\{\|\mA\|_{F}^2+  \eta\|\mA\|_{\mGamma}^2\big\}^{1/2}
	= \|(\mI+\eta\mGamma)^{1/2} \mA\|_F,
\end{equation}	
where $\|\mA\|_{\mGamma}^2 := \tr(\mA\trans\mGamma\mA)$. Applying this norm $\|\cdot\|_{\eta}$ to our eigenfunction estimation problem, we have
\begin{equation}\label{eqn:new-norm}
	\|\mU - \bmU\|_{\eta}^2  =
	\sum_{r=1}^{R} \big\{ V(\psi_r - \bar{\psi}_{r})+ \eta J(\psi_r - \bar{\psi}_{r})\big\},
\end{equation}
where $\psi_r(\cdot) = \vb\trans(\cdot) \vu_r$ and $\vu_r$ is the $r$-th column of $\mU$, and similarly for $\bar{\psi}_r$.

The norm $\|\cdot\|_\eta$ can be employed locally as a metric for the tangent space of the Stiefel manifold St$(R,K)$ at $\bar{\mU}$;  and therefore  $\|\mDelta_{u1}\|_{\eta} + \|\mDelta_{d1}\|_F$ as a metric for the product manifold $\sM$. For $\delta >0$, define a neighborhood $\overline{\sN}(\delta)$ of the optimal parameter $(\bmU, \bmD)\in \mathcal{M}$ inside the tangent space 
$\sT_{(\bmU,\bmD)}\sM  = \sT_{\bmU} \text{St}(R, K) \times \sT_{\bmD} \mathbb{D}_{+} $ as
\begin{equation} \label{eqn:tangentDeltaNeighbor}
	\overline{\sN}(\delta) := \big\{(\mDelta_{u1}, \bmD\mDelta_{d1}) \in \sT_{\bmU} \text{St}(R, K) \times \sT_{\bmD} \mathbb{D}_{+}:\; \|\mDelta_{u1}\|_{\eta} + \|\mDelta_{d1}\|_F \le  \delta\big\}.
\end{equation}
For a fixed   $\mDelta_{d1}$
with $r = \|\mDelta_{d1}\|_F < \delta$,  the slice of 
$\overline{\sN}(\delta)$ projected onto the first component $\mDelta_{u1}$ is
\begin{equation} 
	\overline{\sN}(\delta,\mDelta_{d1}) := \{ \mDelta_{u1} \in \sT_{\bmU} \text{St}(R, K):  \|\mDelta_{u1}\|_{\eta}   \le  \delta - r\},
\end{equation} 
which is ellipsoid-like  in the tangent space of the Stiefel manifold St$(R,K)$; see Figure~\ref{fig:stiefel}.

\begin{figure}
	\centering
	\includegraphics[width=0.6\textwidth]{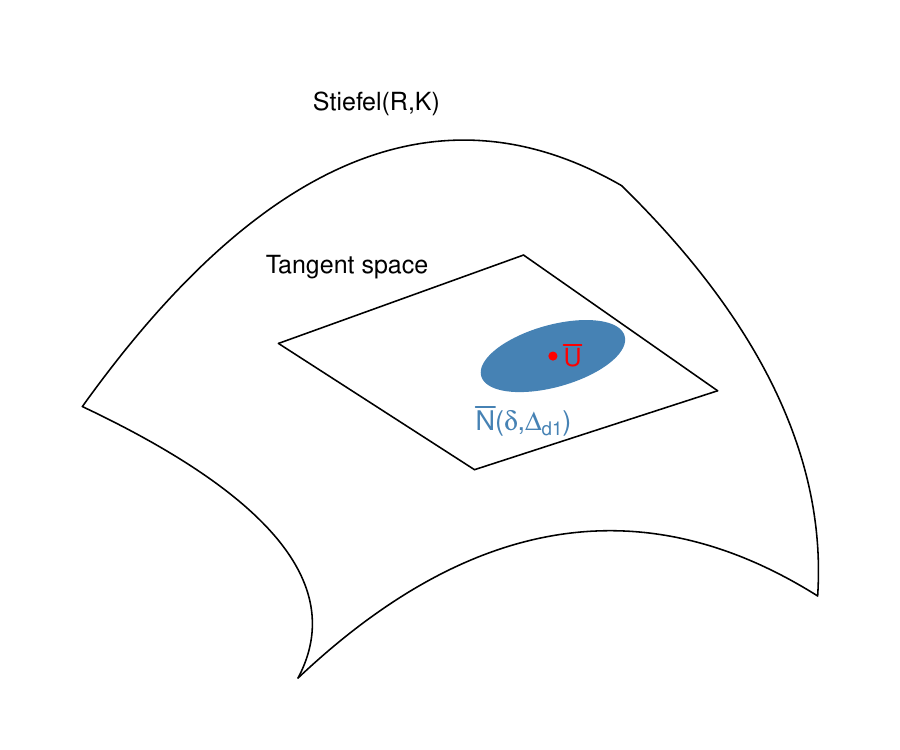}
	\caption{The metric $\|\mDelta_{u1}\|_{\eta} + \|\mDelta_{d1}\|_F$ induces an ellipsoid-like subset (the colored region)  in the tangent space of the manifold St$(R,K)$.\label{fig:stiefel}}
\end{figure}

Compared with the metric $\| \mDelta_{u1}\|_F + \|\mDelta_{d1}\|_F$ used in Lemma~\ref{lemma:local:frobequiv},
the new metric $\|\mDelta_{u1}\|_{\eta} + \|\mDelta_{d1}\|_F$
provides an alternative way to characterize the local perturbation $\mW - \bmW$. 
Define 
\[
\mDelta_{u2} :=  \mU  - \bmU - \mDelta_{u1}= \exp_{\bmU}(1, \mDelta_{u1}) - \bmU - \mDelta_{u1}.
\]
By the triangle inequality,
\begin{equation} \label{eqn:localUnormDiff}
	\|\mDelta_{u1}\|_{\eta} -  \|\mDelta_{u2}\|_{\eta} \le
	\|\mU - \bmU\|_\eta \le \|\mDelta_{u1}\|_{\eta} +  \|\mDelta_{u2}\|_{\eta}.
\end{equation}
Based on the exponential mapping~\eqref{eqn:stiefelExpMap}, we have
\begin{align*}
	\mDelta_{u2} 
	=&\begin{pmatrix}
		\bmU \,& \mQ
	\end{pmatrix} 
	\big[\exp\big(\mS \big) -\mI - \mS\big]
	(\mI_R  ~~ \vzero)\trans
	=\begin{pmatrix}
		\bmU &  \mQ
	\end{pmatrix}
	\mS^2
	\big( \mI/2! + \mS/3!+\ldots\big)
	(\mI_R  ~~ \vzero)\trans.
\end{align*}
Therefore, using the result that $\|\mS\|_F \le 2\|\mDelta_{u1}\|_F \le 2\|\mDelta_{u1}\|_{\eta}$,
\begin{align}
	\|\mDelta_{u2}\|_{\eta}& \le \big\|
	(\mI + \eta\mGamma)^{1/2}
	\begin{pmatrix}
		\bmU\, & \mQ
	\end{pmatrix}  \mS \big\|_{F} \|\mS\|_F \exp(\|\mS\|_F) / 2 \nonumber\\
	& \le \big\|
	(\mI + \eta\mGamma)^{1/2}
	\begin{pmatrix}
		\bmU \,& \mQ
	\end{pmatrix}  \mS \big\|_{F} \|\mDelta_{u1}\|_{\eta} \exp(2\|\mDelta_{u1}\|_\eta) .
	\label{eqn:deltau2:bound}
\end{align}
On the other hand, it holds that
\begin{align}
	\big\| (\mI + \eta\mGamma)^{1/2}
	\begin{pmatrix}
		\bmU \,& \mQ
	\end{pmatrix}  \mS \big\|_{F} &= 
	\big\| (\mI + \eta\mGamma)^{1/2} \begin{pmatrix}
		\mDelta_{u1}\; &\; -\bmU\mR\trans\end{pmatrix}  \big\|_{F}  \nonumber \\
	&\le \|\mDelta_{u1} \|_{\eta} + \|\bmU\|_{\eta} \|\mR\|_F 
	\le  \|\mDelta_{u1} \|_{\eta} (1+  \|\bmU\|_{\eta}). \label{eqn:uqs:bound} 
\end{align}
By Proposition~\ref{prop:approximationerror:eigenfun},
$\|\bmU\|_{\eta}$ is bounded when $\eta^{1/2} R^{1/2} K^{(q-p)_+}$ is bounded.
Combining~\eqref{eqn:deltau2:bound} and~\eqref{eqn:uqs:bound},  we obtain that, 
when $\| \mDelta_{u1}\|_{\eta}$ is sufficiently small,
\begin{equation} \label{eqn:du21}
	\|\mDelta_{u2}\|_{\eta} \lesssim \| \mDelta_{u1}\|_{\eta}^2.
\end{equation}
The above discussion together (\eqref{eqn:localUnormDiff} and \eqref{eqn:du21}) with Lemma~\ref{lemma:local:frobequiv} directly proves the following result.

\vspace{-0.5em}
\begin{lemma} \label{lemma:controlTangent}
	Suppose $\eta^{1/2} R^{1/2} K^{(q-p)_+}$ is bounded. For $(\mDelta_{u1},\bmD\mDelta_{d1})
	\in\overline{\sN}(\delta)$, let  
	$$
	\mU = \exp_{\bmU}(1,\mDelta_{u1}), ~~ \mD = \exp_{\bmD}(1,\bmD\mDelta_{d1}), ~~\text{and}~~ \mW = \mU\mD\mU\trans.
	$$ 
	Then, for $\delta$ small enough, it holds that
	\begin{equation} \label{eqn:newnormequiv}
		\|\mU - \bmU\|_{\eta} \asymp \|\mDelta_{u1}\|_{\eta} ~~\text{and}~~
		\|\mW - \bmW\|_F^2 +\eta \|\mU - \bmU\|_{\mGamma}^2 \asymp  \|\mDelta_{u1}\|_{\eta}^2 + \|\mDelta_{d1}\|_F^2.
	\end{equation}
\end{lemma}

\subsection{Supremum of empirical process}\label{sec:manifold:EP}
We  now give an upper bound (in probability) of the supremum of the empirical process  $G(\mU,\mD)$ defined in~\eqref{eqn:define:Eprocess} over a geodesic neighborhood of the optimal parameter $(\bar\mU, \bar\mD)$ on the product manifold $\sM=\mathrm{St}(R, K) \times \mathbb{D}_+$.
Recall that $\overline{\sN}(\delta)$, defined in~\eqref{eqn:tangentDeltaNeighbor}, is a neighborhood of the optimal parameter $(\bmU, \bmD)\in \mathcal{M}$ on  the tangent space 
$\sT_{(\bmU,\bmD)}\sM  = \sT_{\bmU} \text{St}(R, K) \times \sT_{\bmD} \mathbb{D}_{+}$. We map 
$\overline{\sN}(\delta)$ from the tangent space back to the manifold via the exponential mapping to obtain a geodesic neighborhood $\sN(\delta)$ of $(\bmU, \bmD)$ on $\mathcal{M}$,
\begin{equation}\label{eqn:geo_neighborhood}
	\begin{aligned}
		\sN(\delta) : = \exp(1, \overline{\sN}(\delta) )
		=\Big\{(\mU,&\mD): \ \mU = \exp_{\bmU}(1,\mDelta_{u1}),   \\
		&\qquad 
		\mD = \exp_{\bmD}(1,\bmD\mDelta_{d1}),\ (\mDelta_{u1},\bmD\mDelta_{d1})
		\in\overline{\sN}(\delta)\Big\}.
	\end{aligned}
\end{equation}
\begin{proposition}\label{prop:main:esterror}
	Under Conditions~\ref{ass:density} and \ref{ass:moment}, suppose $K^2\ll N$ and $K \delta^2 \ll 1$. Then it holds that
	\begin{align*} 
		\sup_{(\mU, \mD) \in\sN(\delta)} \mathcal{G}(\mU,\mD)  =	O_p\Big(\frac{R^{1/2}\{\eta^{-1/(4q)} \wedge K^{1/2} +1\}}{N^{1/2}} \Big)	\times  \delta + o_p\big(\delta^2\big).
	\end{align*}
\end{proposition}

\begin{proof}
	We first use the general expression for matrix monotone function~\eqref{eqn:monotonephi} to obtain
	\[
	\varphi'(\mC_n)-\varphi'(\bmC_n)
	=\beta \big(\mC_n - \bmC_n\big) +
	\int_{0}^{\infty} \Big\{
	(\xi \mI + \bmC_n)^{-1} - (\xi \mI + \mC_n)^{-1}   \Big\} \intd\mu(\xi).
	\]
	Denote $\mX_n = \mC_n - \bmC_n$ and note the identity
	\begin{align*}
		& (\xi \mI + \bmC_n)^{-1} - (\xi \mI + \mC_n)^{-1}   \\
		& \quad = (\xi \mI + \bmC_n)^{-1} (\mC_n- \bmC_n) (\xi \mI + \bmC_n)^{-1}                                                 
		- (\xi \mI + \bmC_n)^{-1} (\mC_n- \bmC_n) (\xi \mI + \mC_n)^{-1}  (\mC_n- \bmC_n) (\xi \mI + \bmC_n)^{-1} \\
		& \quad =: \mY_{n}(\xi) - \mZ_n(\xi).
	\end{align*}
	Then the empirical process $\mathcal{G}(\mU,\mD)$    can be decomposed as
	\begin{equation}\label{eqn:esterror:part0}
		\begin{aligned}
			\mathcal{G}(\mU,\mD)  & = \frac{\beta}{N}\sum_{n=1}^{N} \frac{1}{M_n^2} \langle\mX_n , \mS_n- \mK_n\rangle
			+  \frac{1}{N} \sum_{n=1}^{N} \frac{1}{M_n^2} \int \langle\mY_n(\xi) , \mS_n- \mK_n\rangle \intd \mu(\xi) \\ 
			& \qquad\qquad-
			\frac{1}{N} \sum_{n=1}^{N} \frac{1}{M_n^2}\int \langle\mZ_n(\xi) , \mS_n- \mK_n\rangle \intd \mu(\xi). 
		\end{aligned}
	\end{equation}
 Proposition~\ref{prop:main:esterror}  follows by controlling the three terms on the right hand side above; details are given in Lemmas~S.9.1--S.9.3, Section S.9 of the Supplementary Material \cite{he2023penalized}. 
\end{proof}

\section{Behavior of the risk function} 
\label{sec:expectedloss}
This section studies the behavior of the penalized risk function in a small neighborhood of the optimal parameter, where the neighborhood is determined using the metric defined in Section~\ref{sec:manifold:localTangent}.

\subsection{Risk function} 
We first study the behavior of the risk function without considering the penalty function. Using \eqref{eqn:divergenceloss:population}, we can write the difference of the risk function at the parameter $(\mU, \mD)$ and at the optimal parameter $(\bmU, \bmD)$ as
\begin{equation}\label{eqn:risk:difference}
	\begin{split}
		\mathcal{L}_\infty(\mU,\mD) - \mathcal{L}_\infty(\bmU,\bmD) 
		& 
		= \frac{1}{N}\sum_{n=1}^N  \frac{1}{M_n^2}
		D_{\varphi}(\bmC_n || \mC_n)  
		- \frac{1}{N}\sum_{n=1}^N  \frac{1}{M_n^2} \langle \varphi'(\mC_n)-\varphi'(\bmC_n),  \mK_n-\bmC_n\rangle.
	\end{split}
\end{equation}
where $\mC_n$ and $\bar\mC_n$ are defined as in \eqref{eqn:modelCovMat}
and \eqref{eqn:oracleCovMat}, respectively.

Recall that $\mW$ and $\bar{\mW}$ are respectively the coefficient matrices of the tensor product spline basis expansion of the covariance function $\mathcal{C}(u,v)$ and $\overline{\mathcal{C} }(u,v)$. Define the squared empirical norm of the difference $\mW - \bmW$ as
\begin{align}
	\|\mW - \bmW\|_N^2: = \frac{1}{N} \sum_{n=1}^{N}  \frac{1}{M_n^2} \| \mB_{n}(\mW - \bmW) \mB_{n}\trans \|_F^2, \label{eqn:epnorm:define} 
\end{align}
where $\mB_n\trans = (\vb(u_{n1}), \ldots, \vb(u_{nM_n}))$. The following result shows that the first term in~\eqref{eqn:risk:difference} is locally strongly convex with respect to the empirical norm. This result is a direct consequence of Lemma~\ref{lemma:basicconvexity}, and its proof is in Section~S.10.1 of the Supplementary Material \cite{he2023penalized}.
\begin{lemma} \label{lemma:convexLower}
	For $(\mU,\mD)$ in a sufficiently small neighborhood of $(\bmU,\bmD)$, it holds that
	\begin{align}
		\frac{1}{N}\sum_{n=1}^N \frac{1}{M_n^2} \mathcal{D}_{\varphi}(\bmC_n || \mC_n)  
		& \ge  (1 /\overline{M}^2) \Big\{\frac{\beta}{2}+
		\frac{C_\mu}{20(C_u+1)} \Big\} \|\mW - \bmW\|_N^2.
	\end{align}
\end{lemma}

The next result bounds the second term in~\eqref{eqn:risk:difference} in terms of the empirical norm and the approximation error of the covariance function due to using the working model.

\begin{lemma} \label{lemma:linearerror1}
	For $(\mU,\mD)$ in a sufficiently small neighborhood of $(\bmU,\bmD)$, it holds that
	\begin{align*}
		&\frac{1}{N}\sum_{n=1}^N  \frac{1}{M_n^2}\big\langle \varphi'(\mC_n)-\varphi' (\bmC_n),  \mK_n-\bmC_n\big\rangle  
		=   O_p\big( \| \mathcal{K}(u,v) - \overline{\mathcal{C} }(u,v)\|_{L_2}\big)  \times \|\mW - \bmW\|_N.
	\end{align*}
\end{lemma}

\subsection{Empirical norm convergence}
\label{sec:expectedloss:epconv}
This section connects the empirical norm with the metric on the tangent space as defined in Section~\ref{sec:manifold:localTangent}. When $N$ is large enough, we expect that the squared empirical norm $\|\mW - \bmW\|_N^2$ is equivalent to 
$
\| \mathcal{C}(u,v) - \overline{\mathcal{C} }(u,v)\|_{L_2}^2= \|\mW - \bmW\|_F^2 \asymp
\| \mDelta_{u1}\|_F^2+ \|\mDelta_{d1}\|_F^2$
with high probability, 
for any $\mW$ in a neighborhood of $\bmW$ over the product manifold $\sM$ as characterized in~\eqref{eqn:local:lowrank1}. However, establishing such equivalence is intricate. 
Inspecting each summand in~\eqref{eqn:epnorm:define}, we find that
\begin{align} \label{eqn:epnorm:n}
  \| \mB_{n}(\mW - \bmW) \mB_{n}\trans \|_F^2 
	&  =  \sum_{i\neq j} \tr\Big\{
	\mQ_{ni} (\mW - \bmW) \mQ_{nj}(\mW - \bmW)\Big\} + \sum_{i=1}^{M_n} \tr\Big\{\mQ_{ni} (\mW - \bmW)  \mQ_{ni} (\mW - \bmW)\Big\}, 
\end{align}
where $\mQ_{ni} = \vb(u_{ni}) \vb\trans(u_{ni})$.  
Analyzing  the first summation in~\eqref{eqn:epnorm:n} is relatively straight-forward. 
For each summand inside the first summation of~\eqref{eqn:epnorm:n},  the two matrices $\mQ_{ni}$ and $\mQ_{nj}$ are independent as $i\neq j$. 
It is easy to see that the operator norm of $\Expect \big[\mQ_{ni}\otimes \mQ_{nj} \big]$ is bounded by the constant $ C_g^2$  by Condition~\ref{ass:density}.  We need only to control the operator norm of the random matrix
$$
\frac{1}{N} \sum_{n=1}^N  \frac{1}{M_n^2}\sum_{i\neq j} \big\{\mQ_{ni}\otimes \mQ_{nj} - 
\Expect \big[\mQ_{ni}\otimes \mQ_{nj}\big] \big\},
$$
which can be done using a matrix concentration inequality.
On the other hand, controlling the magnitude of the  second summation in~\eqref{eqn:epnorm:n} is more difficult. The same $\mQ_{ni}$ matrix appears twice in each summand, and  the operator norm of $\Expect \big[\mQ_{ni}\otimes \mQ_{ni} \big]$ scales with the spline degree of freedom $K$. We exploit the local manifold structure~\eqref{eqn:local:lowrank1} to alleviate the effect of a diverging $K$ and unbounded $\Expect \big[\mQ_{ni}\otimes \mQ_{ni} \big]$. 
Our result is summarized in the following lemma, and the proof can be found in Section~
S.10.3 of the Supplementary Material \cite{he2023penalized}.

\begin{lemma} \label{lemma:quadraticbound}
	Suppose $K,N\to\infty$,	and $K^2\log(K)/N\to 0$. For any 
	$\mW$  in a small neighborhood of $\bmW$ as specified in~\eqref{eqn:local:lowrank1}.
	Then, it holds that
	\begin{align} \label{lemma:eqn:quadraticlower}
		\|\mW - \bmW\|_N^2  \ge
		\frac{c_g^2C_E^2}{4} \big( \| \mDelta_{u1}\|_F^2+ \|\mDelta_{d1}\|_F^2 \big) + O\big(\|\mDelta_{w1}\|_F^4\big),
	\end{align}
	and
	\begin{equation}\label{lemma:eqn:quadraticupper}
		\|\mW - \bmW\|_N^2 			\le   (3+12C_u/\underline{M}) (C_g+1)^2 C_\lambda^2\big( \| \mDelta_{u1}\|_F^2+ \|\mDelta_{d1}\|_F^2 \big) + O\big(K\|\mDelta_{w1}\|_F^4\big),\\		
	\end{equation}
	with probability tending to one. 
\end{lemma}

\subsection{Penalized risk function}
\label{sec:expectedloss:main}
The difference of the penalized risk function at the parameter $(\mU, \mD)$ and at the optimal parameter $(\bmU, \bmD)$ can be written as
\begin{equation}\label{eqn:pen-risk:difference}
	\ell_\infty(\mU,\mD) - \ell_\infty(\bmU,\bmD)
	= \mathcal{L}_\infty(\mU,\mD) - \mathcal{L}_\infty(\bmU,\bmD) 
	+ \mathcal{P}_\eta(\mU) - \mathcal{P}_\eta(\bmU).
\end{equation}
Because $\mathcal{P}_\eta(\mU)$ is quadratic, it holds that
\begin{equation}\label{eqn:pen:diffference}
	\mathcal{P}_\eta(\mU) - \mathcal{P}_\eta(\bmU) = 2\eta \langle\mGamma\bmU, \mU - \bmU\rangle  + \eta \|\mU - \bmU\|_{\mGamma}^2.
\end{equation}
The following result shows that the difference $\ell_\infty(\mU,\mD) - \ell_\infty(\bmU,\bmD) $ can be lower bounded by a quadratic function of  the metric $\|\mDelta_{u1}\|_{\eta} + \|\mDelta_{d1}\|_F$. 

\begin{proposition} \label{prop:objDiff:metric}
	Suppose $K,N\to\infty$,	and $K^2\log(K)/N\to 0$. Then, with probability tending to one,  for $(\mU,\mD)$ in a sufficiently small neighborhood of $(\bmU,\bmD)$, 
	\begin{align*}
		& \ell_\infty(\mU,\mD)  -  \ell_\infty(\bmU,\bmD)  \\
		& \qquad \quad 		\ge
		( C_c/2 - K^{1/2} \omega_R)\big(\|\mDelta_{u1}\|_{\eta}^2 + \|\mDelta_{d1}\|_F^2\big)  -
		Z \big(\|\mDelta_{u1}\|_{\eta} + \|\mDelta_{d1}\|_F\big) 
		-\eta \|\bmU\|_{\mGamma} ,
	\end{align*}
	where  
	$ C_c = \min\big\{1,\;
	\big(\frac{\beta}{2}+	\frac{C_\mu}{20(C_u+1)}  \big) 	\frac{c_g^2C_E^2}{4\overline{M}^2}\big\}( > 0)$ is a constant,
	and $Z$ is a random variable that satisfies
	$Z = O_p( \eta^{1/2} \|\bmU\|_{\mGamma} +  \|\mathcal{K}(u,v) - \overline{\mathcal{C} }(u,v)\|_{L_2}  ). $
\end{proposition} 

\section{Proof of Theorem~\ref*{thm:main}} \label{sec:theo:mainResult}

Using the tools developed in Sections \ref{sec:manifold} and \ref{sec:expectedloss}, we give the proof of the main  Theorem~\ref{thm:main}.
\begin{proof}
	We allow $\omega_R>0$ below so that the result presented in Remark~\ref{remark:moving-truth} can be obtained using the same argument. Set
	\begin{equation} \label{eqn:defineDeltaN}
		\tau_N =R^{1/2} \{N^{-1/2}(\eta^{-1/(4q)} \wedge K^{1/2}) + \eta^{1/2}K^{(q-p)_+} + K^{-\zeta}\} + \omega_{R}.
	\end{equation}
	Consider a neighborhood of zero in the tangent space of the product manifold $\mathcal{M}$ at $(\bmU,\bmD)$
	\begin{equation} \label{eqn:thm:asypvariance:boundary}
		\overline{\sN}( a\tau_N) = \{(\mDelta_{u1}, \bmD\mDelta_{d1}) \in \sT_{\bmU} \text{St}(R, K) \times \sT_{\bmD} \mathbb{D}_{+}:\; \|\mDelta_{u1}\|_{\eta} + \|\mDelta_{d1}\|_F \le a\tau_N\},
	\end{equation}
	with a constant $a$ to be decided later. Use the exponential mapping to map the set in \eqref{eqn:thm:asypvariance:boundary} to the manifold to obtain a geodesic neighborhood of $(\bmU, \bmD)$, denoted as $\sN(a \tau_N) = \exp\big(1, \overline{\sN}( a\tau_N) \big)$.
	We show that there is a local minimizer in $\sN( a \tau_N)$, denoted as $(\hat{\mU},\hat{\mD})$, for the problem \eqref{eqn:firstOptimization} of the main paper with high probability, as $N \to\infty$. To this end, we only need to show that $\ell(\mU,\mD)>\ell(\bmU,\bmD)$
	for all  boundary points
	\[
	(\mU,\mD)\in \partial \sN(a \tau_N) = \exp\big(1,\partial \overline{\sN}(a \tau_N)\big)
	\]
	with high probability, as $N\to\infty$.
	
	Consider $(\mU,\mD)\in \partial \sN(a \tau_N)$. Write
	$\mU = \exp_{\bmU}(1,\mDelta_{u1})$, $\mD = \exp_{\bmD}(1,\bmD\mDelta_{d1})$, and
	$$(\mDelta_{u1}, \bmD\mDelta_{d1}) \in \sT_{\bmU} \text{St}(R, K) \times \sT_{\bmD} \mathbb{D}_{+}.$$ Then
	$\|\mDelta_{u1}\|_{\eta} + \|\mDelta_{d1}\|_F = a\tau_N$. Below we use $c_1, c_2, \ldots$ to denote appropriate constants.
	Observe that
	\begin{equation}\label{eqn:risk:decomp}
		\ell(\mU,\mD)-\ell(\bmU,\bmD)
		= \ell_\infty (\mU,\mD)-\ell_\infty (\bmU,\bmD) +\mathcal{G}(\mU,\mD).
	\end{equation}
	By~Proposition~\ref{prop:approximationerror:eigenfun}(ii) and \eqref{eqn:penalty} of the main paper, it holds that $\|\bmU\|_{\mGamma}\le c_1 R^{1/2} K^{(q-p)_+}$.
	Thus, Proposition~\ref{prop:objDiff:metric} of the main paper implies that
	\begin{equation*}
		\begin{split}
			\ell_\infty (\mU,\mD)-\ell_\infty (\bmU,\bmD) 
			& \ge c_2 \big(\|\mDelta_{u1}\|_{\eta}^2 + \|\mDelta_{d1}\|_F^2\big)  -
			Z \big(\|\mDelta_{u1}\|_{\eta} + \|\mDelta_{d1}\|_F\big)  -\eta \|\bmU\|_{\mGamma} \\
			& \ge c_2 a^2 \tau_N^2  -
			Z a \tau_N - R^{1/2} \eta K^{(q-p)_+},
		\end{split}
	\end{equation*}
	where $Z$ is a random variable which, by applying Propositions~\ref{prop:approximationerror:eigenfun} and \ref{prop:approximationerror:covfun} of the main paper, satisfies
	\[
	Z = O_P(R^{1/2} \eta^{1/2}K^{(q-p)_+} + R^{1/2}K^{-\zeta} + \omega_{R}) = O_P(\tau_N).
	\]
	Therefore, with high probability, $Z \leq  c_3 \tau_N$, and
	\begin{equation}\label{eqn:risk:decomp1}
		\ell_\infty (\mU,\mD)-\ell_\infty (\bmU,\bmD)
		\ge c_2 a^2 \tau_N^2  -
		c_3 a \tau_N^2    -  \eta^{1/2} \tau_N.
	\end{equation}
	On the other hand, Proposition~\ref{prop:main:esterror} of the main paper implies
	\[
	\mathcal{G}(\mU,\mD) \geq - O_p\Big(\frac{R^{1/2}\{\eta^{-1/(4q)} \wedge K^{1/2} +1\}}{N^{1/2}} \Big)	\times  a \tau_N - o_p\big(a^2 \tau_N^2\big),
	\]
	and thus, with high probability,
	\begin{equation}\label{eqn:risk:decomp2}
		\mathcal{G}(\mU,\mD) \geq - c_4 \tau_N\times a \tau_N - (c_2/2) a^2 \tau_N^2.
	\end{equation}
	Let $a = 2(c_3 + c_4 + 1)/c_2 +1 $ in \eqref{eqn:thm:asypvariance:boundary}.
	Combining \eqref{eqn:risk:decomp}--\eqref{eqn:risk:decomp2}, we obtain that
	\[
	\ell(\mU,\mD)-\ell(\bmU,\bmD) \geq a \tau_N^2 (c_2 a/2 - c_3 - c_4) - \eta^{1/2} \tau_N
	\ge a \tau_N^2 - \eta^{1/2} \tau_N > 0.
	\]
	This completes the proof of existence of the local estimator $(\hat{\mU},\hat{\mD})$ in $\sN( a \tau_N)$.
	
	Since $R$ is fixed, using \eqref{eqn:new-norm} and \eqref{eqn:newnormequiv} of the main paper, the above result implies that
	\begin{equation*}
		\|\hat{\psi}_{r} - \bar{\psi}_{r}\|
		+  \eta^{1/2} J^{1/2}\big( \hat{\psi}_{r} - \bar{\psi}_{r} \big)
		\asymp \| \hat{\mU} - \bmU\|_{\eta}
		\asymp \| \hat{\mDelta}_{u1} \|_{\eta}  \leq a \tau_N.
	\end{equation*}
	It follows from Proposition~\ref{prop:approximationerror:eigenfun}(i) and Proposition~\ref{prop:approximationerror:eigenfun}(ii) of the main paper that
	\begin{equation}\label{eqn:eigenfun:oracleerror}
		\|\bar{\psi}_{r} - \psi_{0r}\| + \eta^{1/2} J^{1/2}(\bar{\psi}_{r}) = O\big( K^{-\zeta}	+ \eta^{1/2}K^{(q-p)_+} \big) = O(\tau_N).
	\end{equation}
	Therefore, by the triangle inequality, it holds that
	\begin{equation*}
		\begin{aligned}
			&
			\|\hat{\psi}_{r} - \psi_{0r} \| + \eta^{{1}/{2}} J^{{1}/{2}}(\hat{\psi}_{r} )  \\
			& \qquad \le
			\|\hat{\psi}_{r} - \bar{\psi}_{r}\|
			+ \eta^{{1}/{2}} J^{{1}/{2}}\big( \hat{\psi}_{r} - \bar{\psi}_{r}
			\big) +
			\|\bar{\psi}_{r} - \psi_{0r} \| + \eta^{{1}/{2}} J^{{1}/{2}}(\bar{\psi}_{r} )  =O_P(\tau_n).
		\end{aligned}
	\end{equation*}
	The desired result follows due to $q\leq m$ and $K^{(q-p)_+} = K^{(q-\zeta)_+}$.
\end{proof}

\section{Consistency of the global estimator} \label{sec:consistency}
	
Theorem~\ref{thm:main} particularly exploits the local manifold geometry discussed in Section~\ref{sec:manifold}. 
This technique leads to the main theoretical results concerning the local minimizer around the optimal parameters of the working model, as detailed in Section~\ref{sec:roadmap}. 
On the other hand, the investigation of the asymptotic behavior of the global minimizer of~\eqref{eqn:firstOptimization} poses distinct challenges. This stems from the fact that the general divergence loss does not necessarily exhibit global convexity, and the parameter space subject to manifold constraints is non-convex inside the ambient space. 

In this section, we attempt to address the challenge and provide a partial answer regarding the asymptotic behavior of the global minimizer. We restrict the parameter space of the working model to be a compact set and establish the consistency of the global minimizer. Specifically, the parameter space of the working model is defined as 
\begin{equation} \label{eqn:sup:compact}
	\begin{split}
		\mathbb{S} = \Big\{(\mU, \mD, \sigma_{e}^2)\in \mathrm{St}(R, K) &\times \mathbb{D}_+ \times \bbR_+:\;
		\tr(\mU\trans \mGamma\mU) \le b_0, \\
		&\lambda_r \in [b_1,b_2] \text{ for } r=1,\ldots, R, \text{ and } \sigma_{e}^2 \in[b_1, b_2] \;
		\Big\},
	\end{split}
\end{equation}
where $b_0$, $b_1$, and $b_2$ are fixed positive values satisfying $b_2>b_1>0$ and such that the optimal parameters $(\bar\mU, \bar\mD, \sigma_{0e}^2)$ are in $\mathbb{S}$. 
According to \eqref{eqn:penalty}, the constraint $\tr(\mU\trans \mGamma\mU) \le b_0$ is placed on eigenfuctions, i.e., $\sum_{r=1}^{ R} \int_{\sU} \big\{ \psi^{(q)}_{r}(u)\big\}^2 \intd u \le b_0$. In application, the upper bound $b_0$ and $b_2$ can be chosen as sufficiently large constants, while $b_1$ can be chosen as a sufficiently small positive number close to zero.

On the compact domain $\mathbb{S}$, we are able to establish the uniform convergence of the estimation error and construct the quadratic bounds for the risk function (i.e., expectation of the loss).  We now consider the following global estimator
\begin{equation} \label{eqn:opt:compact}
	(\widehat\mU, \widehat\mD,\widehat\sigma_e^2) = \argmin_{ (\mU, \mD,\sigma_e^2)\in\mathbb{S}}\, \; \ell(\mU,\mD,\sigma_e^2),
\end{equation}
where $\mathbb{S}$ is defined in~\eqref{eqn:sup:compact}. Denote the $r$-th diagonal element of $\widehat\mD$ as $\widehat{\lambda}_r$, and we let $\widehat{\vu}_{r}$ be the $r$-th column of the global estimator $\widehat{\mU}$ in \eqref{eqn:opt:compact}, $1\leq r \leq R$. Correspondingly, the global estimator of the $r$-th eigenfunction is denoted as $\widehat{\psi}_{r}(u) = \vb\trans(u)\widehat{\vu}_{r}$.
The following theorem presents the consistency of the global estimator. Its proof is provided in Section~
S.12 of the Supplementary Material \cite{he2023penalized}.
	
\begin{theorem} \label{prop:consistency:compact}
	Consider the global minimizer of~\eqref{eqn:opt:compact} and suppose the conditions of Theorem~\ref{thm:main}  hold. Then, for $r=1,2,\ldots, R$, we have that
	$$
		|\widehat{\lambda}_r - \lambda_{0r} | = o_p(1), \qquad
		\|\widehat{\psi}_{r} - \psi_{0r} \|^2   = o_p(1),\qquad \text{and }\qquad |\widehat\sigma_{e}^2 - \sigma_{0e}^2| = o_p(1).
	$$
\end{theorem} 

The proof is provided in Section~S.12 of the Supplementary Material \cite{he2023penalized}.  We provide a high level summary of the proof here. Over the compact set $\mathbb{S}$,  we establish the uniform convergence of the empirical process $G(\mU,\mD,\sigma_{e}^2)$ (defined as in~\eqref{eqn:define:Eprocess} and including the unknown parameter $\sigma_{e}^2$),  and construct the quadratic upper and lower bounds for the risk function~\eqref{eqn:divergenceloss:population}.  We further establish that $(1/N) \sum_{n=1}^N 	\mathcal{D}_{\varphi}(\mK_n ||\widehat \mC_n) / M_n^2 = o_p(1)$, where $\mK_n$ is defined in~\eqref{eqn:trueCovMat} and $\widehat\mC_n = \big[\widehat{\CCov}(u_{nj}, u_{nj'})\big]_{jj'} + \widehat\sigma_e^2\mI$ with $\widehat{\CCov}(u,v): =  \vb\trans(u) \widehat\mU\widehat\mD\widehat\mU\trans \vb(v)$, for $n=1,\ldots, N$. 
In words, the average divergence between the true covariance matrices $\mK_n$'s and the estimated covariance matrices $\widehat \mC_n$'s converges to zero. This leads to the consistency of the estimated covariance function, hence the consistency of its eigenfunctions and eigenvalues.
	
\section{Discussion}\label{sec:disc}
In the context of estimating functional principle components, penalized splines have been demonstrated to enjoy numerical and practical advantages in various earlier works \citep{zhou2008joint,Zhou2014,He2018,Ding2022,he2022spline,sang2022}. This work provides the missing theoretical understanding of the penalized splines in this context, filling the gap between theory and practice. 

The rates of convergence presented in this work are obtained when the number $R$ of principal component functions is a fixed value. It is an interesting future research topic to extend this work to allow $R$ to grow with the sample size. Such an extension needs to address a few new challenges. The $r$-th eigenfunction becomes more difficult to estimate as $r\to\infty$.
This is because: i) the eigen-gap assumption (Condition~\ref{ass:eigengap}) is no longer valid in the sense that the $r$-th eigen gap ($\lambda_{0r} - \lambda_{0,r+1}$) shrinks towards zero as $r$ increases;
ii) we may still assume that each eigenfunction $\psi_{0r}$ belongs to the Sobolev space of order $p$, but its Sobolev norm could diverge to infinity as $r\to\infty$, and consequently, the spline approximation error may not be uniformly controlled for all eigenfunctions to be estimated.

Our asymptotic theory on the penalized spline estimation of PC functions focuses on the setting of sparse functional data and assumes that the number of observation points $M_n$ has a fixed upper bound. In theoretical works on other approaches for estimating the PC functions, $M_n$ is allowed to grow with the sample size $N$ \citep{paul2009consistency,li2010uniform,zhang2016sparse}. Extension of our work to this more general asymptotic setting is left for future research.

\begin{acks}[Acknowledgments]
	The authors would like to thank the anonymous referees, the Associate Editor, and the Editor for their constructive comments that improved the quality of this paper. Kejun He (email:~{kejunhe@ruc.edu.cn}) is the corresponding author. 
\end{acks}

\begin{supplement}
\stitle{Supplement to ``Penalized spline estimation of principal components for sparse functional data: rates of convergence''} \sdescription{The Supplementary Material contains: 1) a table summarizing the notations used throughout the work; and 2) all technical proofs in Sections \ref{sec:meth}--\ref{sec:consistency}.}
\end{supplement}

\bibliographystyle{imsart-nameyear.bst}

\end{document}